\documentclass[11pt, reqno]{amsart}

\usepackage{setup}

\newcommand{\cX}{\mathcal{X}}
\newcommand{\cC}{\mathcal{C}}
\newcommand{\cG}{\mathcal{G}}
\newcommand{\cH}{\mathcal{H}}
\newcommand{\cI}{\mathcal{I}}

\newcommand{\cP}{\mathcal{P}}
\newcommand{\cS}{\mathcal{S}}
\newcommand{\cT}{\mathcal{T}}

\newcommand{\bN}{\mathbb{N}}

\newcommand{\bZ}{\mathbb{Z}}
\newcommand{\bR}{\mathbb{R}}
\newcommand{\bQ}{\mathbb{Q}}

\DeclareMathOperator{\E}{E}
\DeclareMathOperator{\ex}{ex}
\DeclareMathOperator{\vect}{vec}
\DeclareMathOperator{\Ima}{Im}

\usepackage[mathlines]{lineno}
\usepackage{etoolbox} 

\newcommand*\linenomathpatch[1]{%
	\expandafter\pretocmd\csname #1\endcsname {\linenomath}{}{}%
	\expandafter\pretocmd\csname #1*\endcsname{\linenomath}{}{}%
	\expandafter\apptocmd\csname end#1\endcsname {\endlinenomath}{}{}%
	\expandafter\apptocmd\csname end#1*\endcsname{\endlinenomath}{}{}%
}
\newcommand*\linenomathpatchAMS[1]{%
	\expandafter\pretocmd\csname #1\endcsname {\linenomathAMS}{}{}%
	\expandafter\pretocmd\csname #1*\endcsname{\linenomathAMS}{}{}%
	\expandafter\apptocmd\csname end#1\endcsname {\endlinenomath}{}{}%
	\expandafter\apptocmd\csname end#1*\endcsname{\endlinenomath}{}{}%
}

\expandafter\ifx\linenomath\linenomathWithnumbers
\let\linenomathAMS\linenomathWithnumbers
\patchcmd\linenomathAMS{\advance\postdisplaypenalty\linenopenalty}{}{}{}
\else
\let\linenomathAMS\linenomathNonumbers
\fi

\linenomathpatchAMS{gather}
\linenomathpatchAMS{multline}
\linenomathpatchAMS{align}
\linenomathpatchAMS{alignat}
\linenomathpatchAMS{flalign}
\linenomathpatch{equation}
\linenomathpatch{numcases}
\linenomathpatch{figure}

\title[Counting sets free of additive
configurations]{Counting subsets of integers free of arithmetic
configurations}

\author[P. Morris]{Patrick Morris}
        \email{pmorrismaths@gmail.com}
\author[M. Ortega]{Miquel Ortega}
\email{miquel.ortega.sanchez-colomer@upc.edu}
\author[J. Rué]{Juanjo Rué}
\address[Morris, Ortega, Rué] {Departament de Matemàtiques and IMTECH, Universitat Politècnica de Catalunya (UPC), Barcelona, Spain}
\address[Rué]{Centre de Recerca
        Matemàtica (CRM), Barcelona, Spain}
        \email{juan.jose.rue@upc.edu}

\thanks{
Researchers of this work are funded through the research projects:
PID2023-147202NB-I00 (COCOA) (MO, JR) and PCI2024-155080-2 (SRC-ExCo) (PM, MO, JR),
 MSCA-RISE-2020-101007705 {\euflag} (RandNet) (MO, JR), by the Spanish State Research Agency, through the Severo Ochoa and Maria de Maeztu Program for Centers and Units of Excellence in R\&D (CEX2020-001084-M) (JR), by the European Union's Horizon Europe Marie Sk{\l}odowska-Curie grant RAND-COMB-DESIGN - project number 101106032 {\euflag} (PM), by a Ram\'on y Cajal fellowship RYC2024-049272-I (PM) and by the FPI grant PRE2021-099120 (MO)}

\begin{document}

\begin{abstract}
        Cameron and Erd\H{o}s \cite{1990.cameron.erdos} asked if the number of
        sets free of arithmetic progressions of length $k$ is
        $2^{r_k(n)(1+o(1))}$, where $r_k(n)$ is
        the maximum cardinality of a $k$-AP-free subset of $\{1, \dots, n\}$.
        Balogh, Liu and Sharifzadeh \cite{2017.bls} made significant progress on
        this question showing that it is $2^{O(r_k(n))}$ for an infinite
        sequence of $n$. We improve their result in two ways. On the one hand,
        we prove that, for $k\geq 5$, the number of $k$-AP-free sets in $[n]$ is
        $2^{r_k(n)(1+o(1))}$ for an infinite sequence of $n$,
        solving the question of Cameron and Erd\H{o}s for infinitely many
        values. On the other hand, we also prove that for $k \geq 3$ and all $n$ the number of $k$-AP-free
        sets in $[n]$ is $2^{O(r_k(n))}$.

        These results are in fact special cases of a general framework that we develop to count families of sets excluding certain arithmetic
        patterns, which applies as long as the corresponding extremal threshold satisfies certain
        Behrend-type lower bounds. As further examples, we get analogous results for solution sets to almost all systems of linear equations as well as counting versions of the multidimensional Szemerédi
        theorem.

\end{abstract}
\maketitle

\section{Introduction}
A fundamental question in extremal combinatorics consists of counting objects
that avoid a certain configuration.
In the arithmetic setting, Cameron and Erd\H{o}s \cite{1990.cameron.erdos} laid
out several questions and conjectures on counting subsets of $[n] = \left\{ 1,
\dots, n \right\}$ avoiding solutions to certain systems of equations, such as
being sum-free, or not containing an arithmetic progression of a given length.
These questions are tightly related to the foundational extremal problem  of
determining the maximum size of a set satisfying such constraints. Indeed, given
a family of subsets $\cX \subset 2^{\bZ}$, letting
\[
        r_{\cX}(n) = \max \left\{ |A| \colon A \subset [n],
        \, 2^A \cap \cX = \varnothing \right\}
\]
be the maximum size of an $\cX$-free set in $[n]$, it holds that
\begin{equation}
        \label{eq:xfree_naive_bounds}
        2^{r_{\cX}(n)} \leq |\left\{A \subset [n] \colon 2^A \cap \cX =
        \varnothing\right\}| \leq \sum_{i = 0}^{r_{\cX}(n)} \binom{n}{i} \leq
        2^{r_{\cX}(n) \log(n) (1+o(1))}.
\end{equation}

Determining which bound is closer to the truth usually requires a
deeper understanding of the behaviour of $\cX$-free sets, since it somehow
measures how much these are clustered together. In this work, we prove two general theorems that allow us to further
the understanding of this problem for a large class of arithmetic configurations that satisfy
certain lower bounds on $r_{\cX}(n)$. This unifies and strengthens several previous results in the area. 
Our model problem will be that of sets free of
arithmetic progressions of length $k$, or $k$-AP-free sets for short.

\subsection{Counting sets without arithmetic progressions}
Sets free of arithmetic progressions have been one of the main areas of study in
additive combinatorics. In particular, one of the cornerstones of the field is
the study of $r_k(n)$, the size of the largest $k$-AP-free set in $[n]$, with
particular focus on the case $k=3$. For example, Szemerédi's \cite{1975.szemeredi}
celebrated theorem states that $r_k(n) = o(n)$.
For general $k$, the best known bounds are
\begin{equation}
        \label{eq:rankin_bound}
        n \exp\left(-C_k \log(n)^{1/(\lceil \log_2 k \rceil)}\right) \lesssim r_k(n) \lesssim n
        \exp\left(-(\log\log n)^{c_{k}}\right),
\end{equation}
for constants $c_k, C_k > 0$. The best lower
bound, up to lower order terms, is still the one achieved by Rankin's
\cite{1960.rankin} generalization of Behrend's \cite{1946.behrend} construction
for $k=3$ (see \cite{2011.obryant} for the best known constant), and the upper
bound has been proved by Leng, Sah, and Sawhney \cite{2024.lss}.
The case $k=3$ has received particular interest, and the current records are
held by Elsholtz, Hunter, Proske, and Sauermann \cite{2024.ehps} for the lower
bound, and Raghavan \cite{2026.raghavan} for the upper bound, after the
breakthrough of Kelley and Meka \cite{2024.kelley.meka}. 

In any case, the precise asymptotics for $r_k(n)$ are
still far from being established, with a particularly large gap in the case $k >
3$. This is what made the following question of Cameron and Erd\H{o}s
\cite{1990.cameron.erdos} on the number of $k$-AP-free sets seem somewhat
elusive.
\begin{question}
        \label{question:cameron_erdos}
        How does the number of $k$-AP-free subsets of $[n]$ compare with
        $2^{r_k(n)}$? In particular, is it $2^{r_k(n)(1+o(1))}$?
\end{question}

However, a surprising result of Balogh, Liu and Sharifzadeh
\cite{2017.bls} showed that a deep understanding of the behaviour of $r_k(n)$ is
not crucial to establish such bounds.
In the following statement, as in later ones, implicit constants are allowed to
depend on the excluded configuration, in this case on $k$.
\begin{theorem}[Theorem 1.2 \cite{2017.bls}]
        \label{theo:previous_theorem}
        For $k \geq 3$, the number of $k$-AP-free subsets of $[n]$ is at most $2^{O(r_k(n))}$
        for infinitely many values of $n$.
\end{theorem}

We strengthen this result in several directions. Our first main result confirms
the question of Cameron and Erd\H{o}s for $k\geq 5$ and infinite values of $n$.
\begin{theorem}
        \label{theo:kaps_right_constant}
        For $k \geq 5$, the number of $k$-AP-free subsets in $[n]$ is 
        $2^{r_k(n)(1+o(1))}$ for infinitely many values of $n$.
\end{theorem}

We believe that this is a significant step towards understanding Question
\ref{question:cameron_erdos}. Indeed, there are several fundamental examples known where the lower bound in \eqref{eq:xfree_naive_bounds} is in fact not the truth. Most prominently, it was shown by Saxton
and Thomason \cite{2015.saxton.thomason} that the number of Sidon sets (sets
avoiding non-trivial solutions to $a+b=c+d$) in $[n]$ is between $2^{(1.16 +
o(1))\sqrt{n}}$ and $2^{(55+o(1))\sqrt{n}}$ (a finer argument by Kohayakawa,
Lee, R\"odl and Samotij \cite{2015.klrs} gives a better
constant of $(\log_2(32e) +o(1))$ in the exponent), while the maximum size of a Sidon set
is known to be $(1+o(1))\sqrt{n}$. Similarly, in the setting of graphs Morris
and Saxton \cite{2016.morris.saxton} showed that the number of $C_6$-free graphs
on $n$ vertices is at least $2^{1.0007\ex(n,C_6)}$ for infinitely many values of
$n \in \bN$ and at most $2^{O(\ex(n,C_6))}$ for all values of $n$, where
$\ex(n,C_6)$ denotes the largest number of edges in an $n$-vertex $C_6$-free
graph. 
Given these results, it was not necessarily expected that
$(1+o(1))$ should be the correct term in the exponent for Question 
\ref{question:cameron_erdos}. Theorem \ref{theo:kaps_right_constant} gives strong evidence that it is. This mirrors several further examples where  $2^{r_{\cX}(n)(1+o(1))}$ \textit{is} known to be the correct count for the number of $\cX$-free sets (for all $n$). 
 For example, Calkin \cite{1990.calkin} showed that this is the case for
 sum-free sets, Kleitman \cite{1969.kleitman} for antichains, and Balogh, Das,
 Delcourt, Liu, and Sharifzadeh \cite{2019.bdlst} for intersecting families.

 Theorem \ref{theo:kaps_right_constant} is still lacking in the
sense that we only obtain the bound for some values of $n$. 
In this direction, our second main result extends Theorem
\ref{theo:previous_theorem} by proving a bound of the same form for all values
of $n$.
\begin{theorem}
        \label{theo:kaps_alln}
        For $k \geq 3$, the number of $k$-AP-free subsets of $[n]$ is at most
        $2^{O(r_k(n))}$.
\end{theorem}

The previous best-known bound for all values of $n$ was also given by Balogh,
Liu and Sharifzadeh \cite{2017.bls}, who showed that for any $c>0$ and $h(n)$
such that $h(n)\leq \min\{n/r_k(n),(\log n)^c\}$ the number of $k$-AP-free
subsets of $[n]$ is at most $2^{O(n/h(n))}$. Whilst our bound matches theirs if
one believes that $r_k(n)\geq n/(\log n)^c$ for some $c>0$, our bound gives a
$2^{O(r_k(n))}$ upper bound no matter what the form of $r_k(n)$ and so parallels
any improvements in bounds on $r_k(n)$. This is the case, for example, with
$r_3(n)$, the upper bound of which since \cite{2017.bls} has been improved to
$r_3(n)\leq n g(n)^{-1}$ where $g(n)\coloneqq\exp(c'(\log n)^{1/6}(\log \log
n)^{-1})$ for some $c'>0$
\cite{2023.bloom.sisask,2024.kelley.meka,2026.raghavan} (the best known bound at
the time of publication of \cite{2017.bls} was due to Bloom \cite{2016.bloom}
and had the form $r_3(n)=O(n(\log \log n)^4/\log n)$). As $g(n)\gg(\log n)^c$
for any $c>0$, the result of Balogh et al.\ \cite{2017.bls} became weaker than
the trivial upper bound in \eqref{eq:xfree_naive_bounds}. Our bound in Theorem
\ref{theo:kaps_alln} recovers a bound that beats the one in
\eqref{eq:xfree_naive_bounds} and shows that the lower bound
\eqref{eq:xfree_naive_bounds} is always the truth up to a constant in the
exponent.

\subsection{Counting sets free of arithmetic configurations}
Both Theorems \ref{theo:kaps_right_constant} and \ref{theo:kaps_alln} are in
fact instances of more general theorems which provide analogous bounds for
solutions to a wide array of sets of equations. We elaborate fully on our
framework in Section \ref{sec:main_thm_stmts} but suffice it to say here that an
analogous result to Theorem \ref{theo:kaps_alln} holds whenever the
corresponding extremal solution-free set satisfies Behrend-type lower bounds,
whilst bounds analogous to those in Theorem \ref{theo:kaps_right_constant} hold
under stronger Rankin-type lower bounds (which hold for $k$-APs when $k\geq 5$
as shown by Rankin \cite{1960.rankin}). Here we discuss two further implications
of our unified framework.

Firstly, using a result of Shapira \cite{2006.shapira}, we show that such bounds
in fact hold for solution sets to \textit{almost all} translation-invariant
systems of linear equations. To make this notion precise, it is useful to
introduce some definitions. As in \cite{2006.shapira}, we call a linear equation
on $k$ unknowns
\[
        a_1 x_1 + \dots + a_k x_k = 0
\]
with $\sum_{i=1}^k a_i = 0$ a \emph{$(k, h)$-equation} if $a_i \in [-h, h]$ for
all $1 \leq i \leq k$. Furthermore, we say a matrix $A \in \bZ^{m \times k}$ is a $(k, h)$-matrix 
if every row corresponds to the coefficients of a $(k, h)$-equation.
We will be concerned with the system of $m$ $(k, h)$-equations defined by $A
\cdot x = 0$, where
$x^T = (x_1, \dots, x_k)$ is the vector of unknowns.
Setting all $x_i$ to be equal yields
a solution to $A \cdot x = 0$. We call solutions of this form
\emph{trivial}\footnote{Our notion of trivial solutions is not
        quite the standard one, but the difference is not relevant for any of
        our applications. See Remark \ref{remark:trivial_sols} for more
details.}, and say a subset $B \subset \bN$ is $A$-free if there is no
non-trivial solution to $A \cdot x = 0$ with $x_1, \dots, x_k \in B$.
Similarly to the previous cases, we define
\[
        r_A(n) = \max \left\{|B| \colon B \subset [n] \text{ $A$-free}\right\}
\]
to be the size of the largest $A$-free subset in $[n]$. Shapira
\cite{2006.shapira} proved a Behrend-type lower bound on $r_A(n)$ for almost
all systems of $(k, h)$-equations with enough equations. This allows us to prove the
following.
\begin{theorem}
        \label{theo:random_systems_infseq}
        Let $h, k, t \geq 3$ be integers satisfying
         $\binom{t+2}{4} \leq k$.
         Then there is a constant $c = c(k)$
        such that, for all but a $c/h$ fraction of the $(k, h)$-matrices $A \in \bZ^{(k-t+1) \times k}$,
        the number of $A$-free sets
        in $[n]$ for infinite values of $n$ is $2^{r_A(n)(1+o_{h, k}(1))}$.
\end{theorem}
We have chosen to make the dependency on the constants explicit in the previous
theorem to remark that $c$ depends only on $k$. Thus when $h\gg k$, we get that all but a negligible fraction of the $(k,h)$-matrices satisfy the conclusion. As before, we also have a
similar result that holds for all values of $n$ with a weaker bound in the
exponent.
\begin{theorem}
        \label{theo:random_systems_alln}
        Let $h, k, t \geq 3$ be integers satisfying
        $\binom{t}{2} \leq k$. Then there is a constant $c = c(k)$
        such that,
        for all but a $c/h$ fraction of the $(k, h)$-matrices $A \in \bZ^{(k-t+1)\times k}$, the
        number of $A$-free sets in $[n]$ for all values of $n$ is $2^{O_{h, k}(r_A(n))}$.
\end{theorem}

As further corollaries to our main theorems, we give improved counting versions
of the multidimensional Szemerédi theorem, which is concerned with patterns that
generalise arithmetic progressions of finite length to higher dimensions. Given
a dimension $d > 0$ and a finite set $X \subset \bZ^d$, and for the purposes of
this paper, a \emph{copy} of $X$ is a set of the form
\[
        \left\{ b + \lambda\cdot x \colon x \in X \right\},
\]
where $b \in \bR^d$ and $\lambda \in \bR$. Furthermore, we
say the copy is \emph{positive} if $\lambda > 0$, \emph{negative} if $\lambda <
0$, and \emph{trivial} if $\lambda = 0$. We say a set in $\bR^d$ is
\emph{$X$-free} if it only contains trivial copies of $X$. The usual
definition for $X$-free sets in previous literature allows for negative copies
of $X$, but we will only consider this more restrictive
notion, as we discuss in more detail in Remark \ref{remark:negative_copies}
below. The extremal threshold for $X$ is
\[
        r_X(n) = \max
\left\{|A| \colon A \subset [n]^d \text{ $X$-free}\right\}.
\]
The multidimensional Szemerédi theorem states that $r_X(n) = o(n^d)$. Counting
versions of this theorem have already been studied. Indeed, Theorem
\ref{theo:previous_theorem} has been generalised first by Kim \cite{2022.kim}
for counting corner-free subsets of grids, and then by Behague, Hyde, Morrison,
Noel and Wright \cite{2025.bhmnw}, who have recently proved the following.
\begin{theorem}
        \label{theo:behague_etal}
        Let $d$ be a positive integer and let $X \subset \bZ^d$ be a finite set
        such that $|X| \geq 3$. For infinitely many $n \in \bN$, the number of
        subsets of $[n]^d$ free of positive copies of $X$ is $2^{O(r_X(n))}$.
\end{theorem}

We obtain generalisations of Theorem \ref{theo:kaps_right_constant} and Theorem
\ref{theo:kaps_alln} in this setting. On the one hand, we establish bounds
with the correct exponent for sets with $|X| \geq 5$ and infinitely many
values of $n$.
\begin{theorem}
        \label{theo:multi_sz_infseq}
        Let $d$ be a positive integer and let $X \subset \bZ^d$ be a finite set
        such that $|X| \geq 5$. For infinitely many $n \in \bN$, the number of
        $X$-free subsets of $[n]^d$ is $2^{r_X(n)(1+o(1))}$.
\end{theorem}

On the other hand, we prove a bound as in Theorem \ref{theo:behague_etal} for
all values of $n$.
\begin{theorem}
        \label{theo:multi_sz_alln}
        Let $d$ be a positive integer and let $X \subset \bZ^d$ be a finite set
        such that $|X| \geq 3$. The number of $X$-free subsets of $[n]^d$ is
        at most $2^{O(r_X(n))}$.
\end{theorem}
Observe that these theorems generalise our results for $k$-APs, because a set in $\bZ$ is
$k$-AP-free if it does not contain a non-trivial copy of $X=\{1, \dots, k\}$.

\begin{remark}
        \label{remark:negative_copies}
        As we already mentioned, we do not allow for negative copies of $X$ in
        our definition of $X$-free sets, in contrast to the usual definition of
        $X$-free sets in previous work.
        Observe that both notions coincide in the case of symmetric sets
        $X=-X$, such as grids or arithmetic progressions. We take our definition
        because we may then describe $X$-free sets as sets free of solutions to
        a certain translation-invariant system of equations (see Section
        \ref{sec:main_thm_stmts} for more details), and apply our general
        framework to obtain Theorem \ref{theo:multi_sz_infseq} and Theorem
        \ref{theo:multi_sz_alln}. It is also worth mentioning that our methods
        would give a counting result such as Theorem \ref{theo:multi_sz_infseq}
        for sets free of positive copies of $X$, but do not immediately apply
        when trying to establish the analogous result to Theorem
        \ref{theo:multi_sz_alln}.
\end{remark}

\subsection{Further context and main tools}
In the setting of extremal graph theory, counting problems were initiated in the
1970s by Erd\H{o}s, Kleitman, and Rothschild \cite{1976.ekr}. In particular, they
proved that the number of $K_r$-free graphs on $[n]$ vertices is
$2^{(1+o(1))\ex(n, K_r)}$, where $\ex(n, H)$ is the maximum size of an $H$-free
graph (again the lower bound here simply follows from taking subsets of a $K_r$-free graph of maximum size). In the following years, this result was extended to all non-bipartite
graphs \cite{1986.efr} using the regularity method (see, for example,
\cite{2016.morris.saxton} for a more detailed history). However, the
progress on analogous results for bipartite $H$ was much slower. One of the main
reasons for this is that $\ex(n, H) = o(n^2)$ for bipartite $H$, 
which means that tools such as the regularity lemma of Szemerédi cannot be applied.
Another important reason is the fact
that precise asymptotics for $\ex(n, H)$ are not known for most bipartite
graphs, similarly to the case of $k$-APs.

The field was revitalized with the advent of the container method. In
particular, a flurry of results were proved using the theory of hypergraph
containers, a set of tools to count independent sets in hypergraphs
simultaneously developed by Balogh, Morris, and Samotij \cite{2015.bms} and
Saxton and Thomason \cite{2015.saxton.thomason}. We highlight two applications
that are particularly relevant for our work. Firstly, Morris and Saxton
\cite{2016.morris.saxton} established that the number of $C_{2\ell}$-free graphs
on $n$ vertices is $2^{O(n^{1+1/\ell})} = 2^{O(\ex(n, C_{2\ell}))}$. In fact, they also established a much finer
count on the number of $C_{2\ell}$-free graphs of a given size on $n$ vertices.
Secondly, using some of the ideas of Morris and Saxton's result and
inspired by the work of Balogh, Liu and Sharifzadeh on Theorem
\ref{theo:previous_theorem}, Ferber, McKinley and Samotij \cite{2018.fms}
practically settled the problem of counting $H$-free graphs for general bipartite $H$. 
In
particular, they proved the following theorem.
\begin{theorem}
        \label{theo:counting_graphs}
        Let $H$ be an arbitrary graph containing a cycle. Suppose that there are
        positive constants $\alpha$ and $A$ such that $\ex(n, H) \leq
        An^{\alpha}$ for all $n$. Then there exists a constant $C$ depending
        only on $\alpha, A$ and $H$ such that for all $n$, the number of
        $H$-free graphs on $n$ vertices is at most $2^{Cn^{\alpha}}$.
\end{theorem}
This proves that the number of $H$-free graphs is $2^{O(\ex(n, H))}$ if $\ex(n,
H) = \Theta(n^{\alpha})$ for some constant $\alpha > 0$ for bipartite $H$, a
conjecture that is widely believed to be true. In particular, their result
includes all previously known
cases of counting $H$-free graphs up to a constant in the exponent.

While our proof of Theorem \ref{theo:kaps_right_constant} is similar in spirit
to the original one of Balogh, Liu, and Sharifzadeh \cite{2017.bls}, the one of Theorem \ref{theo:kaps_alln} is closer to that of Theorem \ref{theo:counting_graphs} in
\cite{2018.fms} and incorporates an iterative scheme for shrinking containers which is an idea whose origins go back to the work of Morris and Saxton \cite{2016.morris.saxton}. For both theorems, the  main innovation is in providing stronger supersaturation theorems which we also prove in an abstract setting in order to be widely applicable. For our first main theorem giving the tight $(1+o(1))$ factor in the exponent, our stronger supersaturation result requires a certain smoothness condition. The Rankin-type lower bounds on the extremal function are crucial in being able to prove such a condition  for an infinite sequence of $n$.

\subsection{Notation}

We use $\ll$ and $\gg$ in the statement of hypotheses to define hierarchies of
variables, which should be read from  right to left. More concretely, if we state that
$\alpha \ll \beta$ in a hypothesis we mean that there is some increasing function $f$ such that whenever $\alpha \leq f(\beta)$, the corresponding conclusions hold. We also
use $f \lesssim g$ and $f = O(g)$ to mean that there exists a constant $C > 0$
such that $f \leq Cg$. In all cases, the implied constants are allowed to
depend on $\cX$, the family of configurations we forbid, and on $d$, the dimension
of the ambient space $[n]^d$. In particular, in the
case of $k$-APs, we treat $k$ as a constant. Finally, we
occasionally use $f(n) = o(1)$ to denote a function that tends to $0$ as $n$
grows. We omit floors and ceilings when they are not necessary. 

We also use some notation related to hypergraphs.
A
hypergraph $\cH$ over a set of vertices $V$ is a subset of $\cP(V)$, and it is
$r$-bounded if all the hyperedges $E \in \cH$ have size $|E| \leq r$.
Given a hypergraph $\cG$, we write
\[
        \langle \cG \rangle \coloneqq \bigcup_{E \in \cG} \left\{ F \subset
                V(\cG) \colon E \subseteq F \right\}
\]
for the up-set generated by $\cG$. We say that $\cG$ covers a hypergraph $\cH$
if $\cH \subseteq \langle \cG \rangle$. For any $p \in [0, 1]$, we define the
\emph{$p$-weight} of a hypergraph $\cG$ as
\[
        w_p(\cG) \coloneqq \sum_{E \in \cG} p^{|E|}.
\]
\section{Main theorem statements}
\label{sec:main_thm_stmts}
In the previous section, we stated our main counting results for $k$-AP-free sets,
$X$-free sets, or sets free of non-trivial solutions to certain systems of linear
equations.
As we already mentioned, these are all consequences of more general theorems. In
order to set these up, we need to introduce a flexible framework that is able to
describe patterns such as the one in the multidimensional Szemerédi theorem and
systems of linear equations at the same 
time. A way of doing this is to consider systems of equations where unknowns
take values in $\bZ^d$, or in $G^d$ for other groups $G$. 

To do so, we need appropriate notation. Given $k$ variables $x_1, \dots, x_k
\in G^d$ with $x_i^T = (x_i^1, \dots, x_i^d)$, we write $\vect(x_1, \dots, x_k)^T = (x_1^1, \dots, x_1^d, \dots, x_k^1,
\dots, x_k^d)$ for the vector of the corresponding components. In an abuse of
notation, given a vector of components $(v_1, \dots, v_{kd})$, we also write $v^j_i \coloneq
v_{(i-1)d + j}$ for the $j$-th component of the $i$-th variable.

\begin{definition}
        Let $m, k, d > 0$ be integers. We say the matrix $A
        \in \bZ^{m \times kd}$ of $m$ rows is \emph{translation-invariant in dimension
        $d$} if $\sum_{i \in [k]} a^j_i = 0$
        for every row $a$ in $A$ and component $1 \leq j \leq d$. 

        For an abelian group $G$, the \emph{system of equations associated to $A$} is
        \begin{equation}
                \label{eq:linear_config}
                A \cdot \vect(x_1, \dots, x_k) = 0,
        \end{equation}
        where $x_1, \dots, x_k$ are variables taking values in $G^{d}$. 
\end{definition}

\begin{remark}
    We will often refer to a matrix $A\in \bZ^{m\times kd}$ as simply being
\textit{translation-invariant}, suppressing the dimension $d$, which will be
clear from the context. \end{remark}

Observe that, if $A \in \bZ^{m\times kd}$ is a translation-invariant matrix,
then the associated system of equations is translation-invariant in the sense
that
        \[
                \vect(x_1, \dots, x_k) \in \ker A \iff \vect(x_1+x, \dots,
                x_k+x) \in \ker A
        \]
        for all $x, x_1, \dots, x_k \in G^d$.
\begin{definition}
        Let $m, k, d > 0$ be integers, $G$ an abelian group, and let $A \in \bZ^{m
        \times kd}$ be a translation-invariant matrix. We say a  solution $(x_1, \dots,
        x_k)$ to \eqref{eq:linear_config} is \emph{trivial} if $x_1 = \dots =
        x_k$. We also define
        \[
                \cS_A = \left\{\{x_1, \dots, x_k\} \subset G^d \colon \vect(x_1,
                \dots, x_k) \in \ker A \text{ and not all $x_i$ are equal} \right\}
        \]
        to be the family of underlying sets associated to non-trivial solutions of
        \eqref{eq:linear_config}. We say a set $B \subset G^d$ is $A$-free if it
        is $\cS_A$-free, and, when $G = \bZ$, we write
        \[
                r_A(n) = \max \left\{ |B| \colon B \subset [n]^d
                \text{ $A$-free}\right\}
        \]
        for the corresponding extremal threshold function. This generalises the
        case $d=1$, which we already defined in the introduction. Observe that the value of $r_A(n)$
        depends on $d$, which should always be clear from the context.
\end{definition}
\begin{remark}
        The role of $d$ may initially be confusing. If one takes $d=1$ in
        \eqref{eq:linear_config}
        it becomes a usual translation-invariant system of
        equations. Taking larger $d$ allows us to describe copies of $X \subset
        \bZ^d$ in the sense of the multidimensional Szemerédi theorem. Note that
        the system \eqref{eq:linear_config} is a system of equations  at the component level, that is on $x_1^1, \dots, x_1^d, \dots, x_k^1,
\dots, x_k^d$, yet translation-invariance is with respect to the $d$-dimensional variables $x_1,\ldots,x_k$.
\end{remark}
\begin{remark}
        \label{remark:trivial_sols}
        For certain systems of equations, our definition of trivial solutions
        is non-standard. For example, with $d=1$, the equation
        \[
                x_1 - x_2 = x_3 - x_4
        \]
        with $x_i \in \bZ$ (namely, the equation defining Sidon sets) is usually understood to have trivial solutions $(x, x, y,
        y)$ for all $x, y \in \bZ$.

        The standard definition for systems of equations associated to $A \in
        \bZ^{m \times kd}$ would be the following. Suppose there exists a
        partition $[k] = B_1 \sqcup \dots \sqcup B_g$ such that 
        \begin{equation}
                \label{eq:partition_trivial}
                \sum_{i \in B_t} c_i^j = 0
        \end{equation}
        for all rows $c$ in $A$, components $1 \leq j \leq d$ and
        indices $1 \leq t \leq g$. Setting $x_i = x_j$ for every $i, j$ belonging
        to the same set $B_t$ gives a solution to $A \cdot x = 0$. All solutions
        arising in this manner would be called trivial. This is the definition that is
        used in \cite{2006.shapira}, for example (see also \cite{2017.Rueetal}).

        However, for technical reasons we are interested in considering trivial
        solutions as we defined them, that is, such that all $x_i$ are equal. In
        the applications we are concerned with this distinction is not relevant.
        In the case of the multidimensional Szemerédi theorem, given two points
        of a copy of $X$, the whole copy is determined. Hence, if $x_1, \dots,
        x_k$ is a copy of $X$ and $x_i = x_j$ for $i \neq j$ then all $x_i$ are
        equal, which is a trivial solution in our sense. As for random systems
        of equations, a union bound  shows that there exists a constant $c =
        c(k)$ such that, for all but $c/h$ of the $(k, h)$ matrices $A \in
        \bZ^{m\times k}$, $A$ does not admit a partition satisfying
        \eqref{eq:partition_trivial}.
\end{remark}
In Section \ref{sec:applications} we see how arithmetic progressions,
copies of $X \subset \bZ^d$ as in the multidimensional Szemerédi theorem, or solutions to 
systems of linear equations as in Theorem \ref{theo:random_systems_infseq} are all instances of solutions to
systems of equations such as the one in \eqref{eq:linear_config}. Hence, the
extremal threshold function $r_{A}(n)$ is a
generalisation of the extremal threshold functions defined so far.
From this observation, as we will see in detail later, it follows that
Theorem \ref{theo:kaps_right_constant}, Theorem \ref{theo:multi_sz_infseq} and
Theorem \ref{theo:random_systems_infseq} are all consequences of the following statement, where
implicit constants are allowed to depend on $A$.
\begin{theorem}
        Let $m, d > 0$ and $k \geq 3$ be integers and let $A \in \bZ^{m\times kd}$ be a
        translation-invariant matrix. Suppose that
        \[
                r_A(n) \geq n^d\exp(-O(\log(n)^{\alpha}))
        \]
        for some $0 < \alpha < 1/2$.
        Then the number
        of $A$-free sets in $[n]^d$ is at most $2^{(1+o(1))r_A(n)}$ for infinite
        values of $n$.
        \label{theo:inf_seq_bounds}
\end{theorem}
For example, Theorem \ref{theo:kaps_right_constant} follows from
Theorem \ref{theo:inf_seq_bounds} using Rankin's lower
bound \eqref{eq:rankin_bound} on $r_k(n)$ and the fact that
$k$-APs may be described by a system of $k-2$ translation-invariant
equations. Analogously, Theorems
\ref{theo:kaps_alln}, \ref{theo:multi_sz_alln} and
\ref{theo:random_systems_alln} are instances of the following general theorem.
\begin{theorem}
        \label{theo:main_alln}
        Let $m, d > 0$ and $k\geq 3$ be integers and let $A \in \bZ^{m \times kd}$ be a
        translation-invariant matrix. Suppose that
                \[
                        r_A(n) \geq n^d\exp(-O(\log(n)^{\alpha}))
                \]
                for some $0 < \alpha < 1$.
        Then the number
        of $A$-free sets in $[n]^d$ is at most $2^{O(r_A(n))}$. 
\end{theorem}

We prove Theorem \ref{theo:inf_seq_bounds} in Section \ref{sec:inf_seq} and Theorem \ref{theo:main_alln} in Section \ref{sec:alln}. In Section \ref{sec:applications} we then derive the applications stated in the introduction and in Section \ref{sec:final-remarks} we provide some concluding remarks. 
\section{Tools}
Let us first collect several standard results that we use in different sections
of the paper. The reader is welcome to skip these and only
come to them when in need.
In the first place, we need an elementary bound on the growth of
the extremal function $r_{A}(n)$.
\begin{lemma} 
        \label{lemma:upper_multiplicative}
        Let $m', k, d > 0$ be integers, and 
        let $A \in \bZ^{m'\times kd}$ be a translation-invariant matrix.
        The extremal threshold $r_A(n)$ satisfies the inequality
        \begin{equation}
                \label{eq:upper_multiplicative}
                r_A(mn) \leq m^d r_A(n)
        \end{equation}
        for all integers $m, n > 0$.
\end{lemma}
\begin{proof}
        Let $B \subset [mn]^d$ be a set with no non-trivial solution to
        \eqref{eq:linear_config}. We split the grid $[nm]^d$ into $m^d$ disjoint
        boxes of size $n^d$ each. Formally, we consider the partition 
        \[
                [nm]^d = \bigsqcup_{x \in [0, m-1]^d} (n \cdot x + [n]^d).
        \]
        Let $Q = nx+[n]^d$ be a box in this partition.
        Since containing a solution to
        \eqref{eq:linear_config} is a translation invariant property, $(B\cap Q) -
        nx \subset [n]^d$ is an $A$-free subset of $[n]^d$, so has size at most
        $|B \cap Q| \leq r_A(n)$. Adding over all possible boxes,
        the size of $B$ is at most $m^dr_A(n)$.
\end{proof}

\subsection{Freiman isomorphisms}
It will also be convenient to move sets between different abelian groups while
preserving their additive structure. We use the language of Freiman
homomorphisms (see, for example, \cite[Section 5.3]{2006.tao.vu}), and reprove
some of their properties in our multidimensional setting.
\begin{definition}
        Let $s \geq 2$ be a positive integer and $G, G'$ abelian groups.
        A \emph{Freiman homomorphism} of order $s$, or Freiman $s$-homomorphism,
        between subsets $B \subset G$ and $B' \subset G'$ is a map $\phi \colon B
        \to B'$ such that
        \begin{equation}
                \label{eq:freiman_iso}
                \phi(a_1) + \dots + \phi(a_s) = \phi(b_1) + \dots + \phi(b_s)
        \end{equation}
        whenever $a_1 + \dots + a_s = b_1 + \dots + b_s$ for $a_1, \dots a_s,
        b_1, \dots, b_s \in B$.
        A bijection $\phi \colon B \to B'$ is a \emph{Freiman isomorphism} of order $s$ if both
        $\phi$ and $\phi^{-1}$ are Freiman homomorphisms of order $s$.
\end{definition}

In particular, we need a morphism to flatten out product sets at the cost of
increasing the size of the ambient set by a constant factor.
\begin{lemma}
        \label{lemma:freiman_product}
        Given an integer $s \geq 2$,
        there exists a Freiman $s$-isomorphism that maps $[m_1] \times \dots
        \times[m_r]$ to a subset of $[s^{r-1}m_1\cdots m_r]$.
\end{lemma}
\begin{proof}
        We prove the case $r=2$, since larger cases follow by applying the same argument
        repeatedly. We define
        \begin{align*}
                \phi \colon [m_1] \times [m_2] &\to [sm_1 m_2] \\
                (a, b) &\mapsto asm_2 - b.
        \end{align*}
        We claim that $\phi$ is a Freiman $s$-isomorphism when we
        restrict its codomain to its image. 
        Since $\phi$
        is linear it is a Freiman homomorphism of any order.
        Now, suppose that
        \[
                \phi(a_1, b_1) + \dots + \phi(a_s, b_s) = 
                \phi(a_1', b_1') + \dots + \phi(a_s', b_s')
        \]
        for $a_1, \dots, a_s,a_1', \dots a_s' \in [m_1]$ and
        $b_1, \dots, b_s,b_1', \dots b_s' \in [m_2]$. Then
        \begin{equation}
                \label{eq:freiman_iso_prod}
                (a_1 + \dots + a_s)sm_2 - (b_1 + \dots + b_s) =
                (a_1' + \dots + a_s')sm_2 - (b_1' + \dots + b_s').
        \end{equation}
        Taking this equality modulo $sm_2$ implies that
        \[
                b_1 + \dots + b_s \equiv b_1' + \dots + b_s' \mod sm_2,
        \]
        so $sm_2 \mid (b_1 + \dots + b_s - b_1' - \dots - b_s')$. Finally, since
        $|b_1 + \dots + b_s - b_1' - \dots - b_s'| < sm_2$, it follows that it
        must be equal to $0$. In other words, we have
        \[
                b_1 + \dots + b_s = b_1' + \dots + b_s'.
        \]
        After subtracting this from \eqref{eq:freiman_iso_prod} and dividing by
        $sm_2$ we obtain $a_1 + \dots + a_s = a_1' + \dots + a_s'$, so
        $\phi^{-1}$ is a Freiman $s$-homomorphism when restricted to
        $\phi([m_1]\times[m_2])$, which finishes our proof.
\end{proof}

Finally, we also need that Freiman isomorphisms preserve the property of being
$A$-free.
\begin{lemma}
        \label{lemma:freiman_afree}
        Let $m, k, d > 0$ be integers, let $G, G'$ be abelian groups, and
        let $A \in \bZ^{m \times kd}$ be a translation-invariant matrix. There exists $s > 0$ such
        that, for any set $Z \subset G$ 
        and Freiman $s$-isomorphism $\phi \colon Z \to Z'$, the product map
        $\phi^{\times d} = (\phi, \dots, \phi)^T$ satisfies
        \[
                A \cdot \vect(x_1, \dots, x_k) = 0 \iff A \cdot
                \vect(\phi^{\times d}(x_1), \dots, \phi^{\times d}(x_k)) = 0
        \]
        for every $x_1, \dots, x_k \in Z^d$.
        In particular, if  $B \subset Z^d$ is
        $A$-free then $\phi^{\times d}(B) \subset Z'^d$ is $A$-free.
\end{lemma}
\begin{proof}
        Let us prove the right-to-left implication, the converse follows from
        applying it to $\phi^{-1}$.
        We claim that $s=\max(2, kd \max_{i, j} |A_{i, j}|)$ suffices for the statement to
        hold. 
        Suppose that $\phi^{\times d}(x_1), \dots, \phi^{\times d}(x_k)$
        satisfy 
        \begin{equation}
                \label{eq:system_freiman}
                A \cdot \vect\left(\phi^{\times d}(x_1), \dots, \phi^{\times
                        d}(x_k)\right) = 0
        \end{equation}
        for $x_1, \dots, x_k \in B \subset Z^d$. Let $c \in \bZ^{kd}$ be a given row of
        $A$, which, by the previous equation, satisfies $c \cdot \vect\left(\phi^{\times d}(x_1),
        \dots, \phi^{\times d}(x_k)\right) = 0$. Let us write $x_i^T = (x_i^1, \dots,
        x_i^{d}$)
        for the components of $x_i$, and let
        $c^i_j \coloneqq
        c_{(i-1)d+j}$ for $1 \leq i \leq k$ and $1 \leq j \leq d$, that is,
        the coefficient in $c$ corresponding to $x_i^j$.
        Then
        \[
                \sum_{\substack{1 \leq i \leq k, \\ 1 \leq j \leq d}} (c_j^i)_+
                \phi(x_i^j) = 
                \sum_{\substack{1 \leq i \leq k, \\ 1 \leq j \leq d}}
                (c^i_j)_{-}
                \phi(x_i^j) \]
                by definition of $\phi^{\times d}$, where $a_+ = \max(0, a)$ and
                $a_- = \max(0,-a)$. Since sums on both sides have at most
        $|c_1| + \dots + |c_{kd}| \leq s$ terms, a Freiman $s$-isomorphism
        preserves them, implying that
        \[
                c \cdot \vect(x_1, \dots, x_k) = \sum_{\substack{1 \leq i \leq
        k, \\ 1 \leq j \leq d}} c^i_j
                x_i^j = 0.
        \]
        Doing this for all rows in $A$ gives $A \cdot \vect(x_1, \dots, x_k) =
        0$. If, furthermore, $B$ is $A$-free, we deduce that $x_1= \dots =  x_k$ and then $\phi^{\times d} (x_1)= \dots =
        \phi^{\times d}(x_k)$. That is, all solutions of \eqref{eq:system_freiman} are trivial, and
        $\phi^{\times d}(B)$ is $A$-free.
\end{proof}

Combining both lemmas allows us to flatten out $A$-free sets living in a
product space.
\begin{corollary}
        \label{cor:freiman_product}
        Let $m, k, d > 0$ be integers, and let $A \in \bZ^{m \times kd}$ be
        a translation-invariant matrix. Then there exists $s > 0$ such that,
        for any set $B \subset ([m_1] \times \dots \times [m_r])^d$, there
        exists a subset $B' \subset [s^{r-1} m_1 \dots m_r]^d$ with $|B'| = |B|$
        such that $B$ is $A$-free if and only if $B'$ is $A$-free.
\end{corollary}
\begin{proof}
        Let $s$ be the integer provided by Lemma \ref{lemma:freiman_afree}, and
        let $\phi \colon [m_1] \times \dots \times [m_r] \to Z \subset [s^{r-1}m_1 \dots
        m_r]$ be the Freiman
        $s$-isomorphism given by Lemma \ref{lemma:freiman_product}. Then
        $B' = \phi^{\times d}(B)$ is $A$-free if and only if $B$ is $A$-free on account of
        $\phi$ being a Freiman $s$-isomorphism and Lemma
        \ref{lemma:freiman_afree}.
\end{proof}

\subsection{Hypergraph containers}
Our proof uses the container method, developed simultaneously by Balogh, Morris,
and Samotij \cite{2015.bms} and Saxton and Thomason \cite{2015.saxton.thomason}
as a tool to count independent sets in hypergraphs.
We will use a recent formulation due to Campos and Samotij
\cite{2024.campos.samotij}.
The use of Campos-Samotij containers is not crucial for our proof, but it
simplifies the presentation. In particular, we will apply the container method
to the hypergraph induced by members of a given family $\cX$ contained in $[n]^d$, and previous
hypergraph container lemmas would require us to build a subhypergraph by
choosing only some of the members of $\cX$ in such a way that the degrees are properly
balanced. Instead, Campos-Samotij containers naturally include this step in their
own framework.
Concretely, we will use the following hypergraph container lemma, which is due
to Campos and Samotij \cite{2024.campos.samotij}.
\begin{theorem}
        \label{theo:cs_containers}
        Let $\cH$ be an $r$-bounded hypergraph with vertex set $V$ and $2 \leq |E| \leq
        r$ for all edges $E \in \cH$. For every $p
        \in (0, 1/(8r^2)]$, there exists a family $\cT \subseteq 2^V$ and
        functions
        \[
                g \colon \cI(\cH) \to \cT \quad \text { and } f \colon \cT \to
                2^V
        \]
        such that:
        \begin{enumerate}[label=(\roman*)]
                \item For each $I \in \cI(\cH)$, we have $g(I) \subseteq I
                        \subseteq f(g(I))$.
                        \label{item:containers_finger}
                \item Each $T \in \cT$ has at most $8r^2 p |V|$ elements.
                        \label{item:fingerprint_size}
                \item 
                        \label{item:containers_cover}
                        For every $T \in \cT$, letting $C \coloneqq f(T)$, there
                        exists a hypergraph $\cG$ on $C$ with
                        \[
                                w_p(\cG) \leq p|C|
                        \]
                        that covers $\cH[C]$ and satisfies $|E| \geq 2$ for all
                        $E \in \cG$.
        \end{enumerate}
\end{theorem}
As is usual, we will refer to the family $\cT$ in the previous theorem as the
family of \emph{fingerprints} of the independent sets of $\cH$, and to the sets in
$f(\cT)$ as \emph{containers} of the independent sets. Notice that 
\cite[Theorem A]{2024.campos.samotij} is stated for uniform hypergraphs, but the
derivation of Theorem A from Theorem E given in the paper yields this stronger
version.

\section{Bounds over an infinite sequence}
\label{sec:inf_seq}
The main goal of this section is proving Theorem \ref{theo:inf_seq_bounds}.
Throughout the section, we
fix dimension $d \geq 1$, the number of variables $k \geq 3$, and a translation-invariant matrix $A \in
\bZ^{m \times kd}$ with its associated system of equations.

We start out with our main supersaturation result, which we prove using a random
sampling argument. Usual supersaturation results state that sets $B$ of size
greater than $r_A(n)$ contain not just one, but many solutions to the system
associated to $A$. In the language of Theorem \ref{theo:cs_containers}, it is
convenient to prove a dual statement, showing that the hypergraph $\cS_A[B]$,
for a set $B \subset [n]^d$ of size greater than $r_{A}(n)$, does not admit a
cheap cover, that is, a cover where $w_p(\cG) \leq p |B|$ for $p$ suitably small.
Combining such a result with item \ref{item:containers_cover} of Theorem
\ref{theo:cs_containers} gives an upper bound on the size of containers from
which we can then derive our counting results.

In order to establish our supersaturation result, we do roughly the following. Given a small potential cover $\cG$ of
$\cS_A[B]$,
we intersect $B$ with a random grid $\mathbf{G} = \mathbf{b} + \mathbf{d} \cdot
[\ell]^d$, where $\mathbf{b}$ is a random basepoint in $[n]^d$ and $\mathbf{d}$
a random common difference in $[n/\ell]$. If the grid is sparse enough, we are
able to show that a small alteration of $\mathbf{G} \cap B$ contains no element
of $\cG$. As long as we are able to ensure that $\mathbf{G} \cap B$ remains
larger than $r_{A}(\ell)$, this will give us a member of $\cS_A[B]$ that is not
covered by $\cG$.
To make this argument precise, we will actually
take $\mathbf{G}$ in a box inside $[n]^d$ to avoid some border issues, and only
consider prime common differences in order to control the number of possible
grids that contain a given pair of elements in $[n]^d$.

We have chosen to state the proposition in somewhat general terms because we believe
it clarifies the proof. However, the reader may want to keep in mind that we
will apply it with $\ell \sim n \exp(-\sqrt{\log n})$ and $\delta$ slowly going to
zero (say, $\delta \sim 1/\log\log n$, but the exact rate is not important). In
particular, the terms depending on $\delta$ and $\log n$ in
\eqref{eq:bound_cover_1}
will be negligible compared to the $(\ell/n)^{1/2k}$ term, and
may be ignored on a first reading. 

\begin{proposition}
        \label{prop:supersaturation_seq}
        Let $n \gg \ell \gg 1/\delta \gg 1$ with $n, \ell$ integers be such
        that
        \begin{equation}
                \frac{r_A(n)}{n^d} \geq \left(1-\frac{\delta}{8}\right)\frac{r_A(\ell)}{\ell^d}
                \label{eq:rk_hypothesis}.
        \end{equation}
        Let $B \subset [n]^d$ be a subset of size $|B| \geq (1+\delta)
        r_A(n)$. Then there is no cover $\cG$ of $\cS_{A}[B]$ with $|E| \geq 2$
        for all $E \in \cG$ such that $w_q(\cG) \leq q|B|$ for
        \begin{equation}
\label{eq:bound_cover_1}
                 \frac{\log(n)}{\delta^2} \left(\frac{\ell}{n}\right)^{1/2k} \ll
                 q \leq 1.
        \end{equation}
\end{proposition}
\begin{proof}
        Suppose there exists such a cover $\cG$, which we may assume to be $k$-bounded. We aim to reach a contradiction
        by finding a member of $\cS_A$ lying in $B$ that is not covered by $\cG$.
        For convenience, we write $t = \sqrt{\ell n}$ for the geometric mean of
        $\ell$ and $n$. Define the box $Q = [t+1, n-t]^d$, which consists of
        removing some small borders from $[n]^d$. We now split the proof in two
        cases, according to how many elements of $B$ lie on the borders $\bar{Q}
        = [n]^{d} \setminus Q$.

Assume first that $|B \cap \bar{Q}| \geq \delta |B|/4$. In this
        case, the set $B$ has density significantly higher than expected on
        $\bar{Q}$, so we look for members of $\cS_A$ not covered by $\cG$ in $B
        \cap \bar{Q}$. To do this, let $\mathbf{Z} \sim (B \cap \bar{Q})_p$ be an independent
        random sample of the elements of $B \cap \bar{Q}$ with probability $p =
        q|B \cap \bar{Q}|/2|B|$. Notice that
        \[
                \E[|\mathbf{Z}|] \geq p |B \cap \bar{Q}|,
        \]
        and, using the fact that all members of $\cG$ have size at least $2$, it follows that
        \[
                \E[|\cG \cap 2^{\mathbf{Z}}|] \leq w_p(\cG) \leq \left(\frac{p}{q}\right)^2
                w_q(\cG) \leq \frac{p^2|B|}{q} \leq \frac{p |B \cap \bar{Q}|}{2}.
        \]
        By the previous two bounds and the first moment method, there exists a
        set $Z \subset B \cap \bar{Q}$ satisfying
        \[
                |Z| - |\cG \cap 2^{Z}| \geq \frac{p|B\cap\bar{Q}|}{2}.
        \]
        Select an arbitrary $x \in e$ for every $e \in \cG \cap 2^{Z}$, and
        define $Z'$ by removing all such $x$ from $Z$. Using first the
        definition of $p$, and then our assumption on the size of
        $|B\cap\bar{Q}|$ and \eqref{eq:bound_cover_1}, we obtain
        \begin{equation}
                \label{eq:bound_z}
                |Z'| \geq \frac{p|B\cap\bar{Q}|}{2} = \frac{q
                |B\cap\bar{Q}|^2}{4|B|} \geq \frac{q\delta^2}{64} |B| \geq
                \left(\frac{\ell}{n}\right)^{1/2k} r_{A}(n).
        \end{equation}
        We claim that $Z'$ contains a
        member of $\cS_A$, which, by construction, must not be covered by $\cG$.
        To prove this, suppose it does not and tile $\bar{Q} = Q_1 \cup \dots \cup Q_r$ with
        $Q_1, \dots, Q_r$ disjoint and every $Q_i$ a translated copy of a box
        $[h_i]^d$, with $1 \leq h_i \leq t$ box $[t]^d$, using a total amount of
        \[
                r \lesssim \frac{|\bar{Q}|}{t^d} = \frac{n^d - (n-2t)^d}{t^d}
                \lesssim \left(\frac{n}{t}\right)^{d-1}
        \]
        copies. If $Z'$ is $A$-free, the translation invariance property of
        solutions to \eqref{eq:linear_config}
        gives that
        \[
                |Z' \cap Q_i| \leq r_{A}(t),
        \]
        which, first adding up over all possible $Q_i$ and then using
        Lemma \ref{lemma:upper_multiplicative} and \eqref{eq:rk_hypothesis}, implies
        that
        \[
                |Z'| \leq r \cdot r_{A}(t) \leq r 
                \left(\frac{t}{\ell}\right)^d r_{A}(\ell) \lesssim
                r\left(\frac{t}{n}\right)^d r_{A}(n) \lesssim
                \frac{t r_{A}(n)}{n}=\left(\frac{\ell}{n}\right)^{1/2} r_A(n),
        \]
        a contradiction with \eqref{eq:bound_z} since $t=(\ell n)^{1/2}$.
        This proves our claim that $Z'$ contains a member of $\cS_A$ not covered
        by $\cG$.

        Hence, we may suppose now that $|B\cap \bar{Q}| \leq \delta |B|/4$, so
        that $|B \cap Q| \geq (1-\delta/4)|B|$. We
        proceed in a similar manner as in the previous case, but our random
        sample will be slightly more complicated. In order to find an uncovered
        member of $\cS_A$, we will intersect $B$ with a random grid
        \[
                \mathbf{G} = \mathbf{b} + \mathbf{p} \cdot [\ell]^d,
        \]
        where $\mathbf{p}$ is a prime chosen uniformly at random among those in
        $[t/\ell] = [(n/\ell)^{1/2}]$ and $\mathbf{b}$ is a basepoint chosen uniformly at random in
        $[1, n-t]^d$. Notice that $\mathbf{G} \subset [n]^d$. Let us first bound the probability
        that a point in $Q$ lies in $\mathbf{G}$.
        The 
        number of choices for $\mathbf{p}$ is $\pi(t/\ell)$, the total number of
        primes in $[t/\ell]$, and the number of choices for $\mathbf{b}$ is
        $(n-t)^d$. Given $x\in Q$, $v \in [\ell]^d$, and $p \in [t/\ell]$, we
        have $x - p\cdot v \in [1,n-t]^d$, on account of the definition of
        $Q$. Hence, there are $\pi(t/\ell)\ell^d$ outcomes of $\mathbf{G}$
        containing $x \in Q$, since all possible outcomes of $\mathbf{p}$ and
        relative positions of $x$ in $\mathbf{G}$ produce a corresponding basepoint
        lying in $[1,n-t]^d$. It follows that
        \begin{equation}
                \label{eq:lb_pg}
                \Pr(x \in \mathbf{G}) = \frac{\pi(t/\ell)\ell^d }{\pi(t/\ell)(n-t)^d}
                \geq \left(\frac{\ell}{n}\right)^d,
        \end{equation}
        for all $x \in Q$.

        On the other hand, given two distinct points $x, y \in Q$, the number of outcomes of $\mathbf{G}$
        containing both is at most $\ell^d \log(n)$. The $\ell^d$ factor
        accounts for the relative position of $x$ in $\mathbf{G}$, and the $\log(n)$
        factor follows from the fact that $\mathbf{p}$ must be a prime divisor of
        $0 \neq (x-y)_i \leq n$ for some $1 \leq i \leq d$, since there are at most $\log(m)$ different prime
        factors of a given integer $m > 6$. From this, together with the prime
        number theorem, we see that
        \begin{equation}
                \label{eq:upper_ppg}
                \Pr(x, y \in \mathbf{G}) \leq \frac{\ell^d \log(n)}{\pi(t/\ell)(n-t)^d}
                \lesssim \log(n)\frac{\ell\log(t/\ell)}{t}
                \left(\frac{\ell}{n}\right)^d \lesssim \log(n)^2
                \left(\frac{\ell}{n}\right)^{d+1/2},
        \end{equation}
        for all distinct $x,y \in Q$. With these bounds in hand, we are ready to
        carry out the central part of the proof. Let $\mathbf{Z} = B \cap \mathbf{G}$ be
        the elements of $B$ lying in our random grid. It follows from
        \eqref{eq:lb_pg} and our assumption on the size of $B \cap Q$ that
        \[
                \E[|\mathbf{Z}|] = \sum_{x \in B} \Pr(x \in \mathbf{G}) \geq \sum_{x \in B
                \cap Q} \Pr(x \in \mathbf{G}) \geq \left(\frac{\ell}{n}\right)^d |B
                \cap Q| \geq \left(\frac{\ell}{n}\right)^d
                (1-\delta/4)|B|.
        \]
        Since all members of $\cG$ have size at least $2$ and at most $k$, from
        \eqref{eq:upper_ppg} it also follows that
        \begin{multline*}
                \E[|\cG \cap 2^{\mathbf{Z}}|] \leq \sum_{e \in \cG} \Pr(e \in \mathbf{G})
                \leq |\cG|
                \log(n)^2\left(\frac{\ell}{n}\right)^{d+1/2} \leq \\
                \leq q^{-(k-1)}|B|\log(n)^2\left(\frac{\ell}{n}\right)^{d+1/2} \leq
                \frac{\delta |B|}{4} \left(\frac{\ell}{n}\right)^{d}.
        \end{multline*}
        Hence, by the two previous bounds and the first moment method there
        exist a grid $G$ and subset $Z = B \cap \mathbf{G}$ satisfying
        \[
                |Z|-|\cG \cap 2^{Z}| \geq \left(\frac{\ell}{n}\right)^d
                \left(1-\frac{\delta}{2}\right)|B| \geq (1+\delta/4) \left(\frac{\ell}{n}\right)^d
                r_{A}(n) \geq \left(1+\frac{\delta}{16}\right)r_{A}(\ell),
        \]
        where we used \eqref{eq:rk_hypothesis} in the last inequality.
        As before, select an arbitrary $x \in e$ for every $e \in \cG \cap
        2^{Z}$, and define $Z'$ by removing all such $x$ from $Z$. Finally,
        notice that $Z' \subset G$ is contained in a grid of size $\ell^d$ and
        has size larger than $(1+\delta/16)r_{A}(\ell)$. By the invariance
        of $\cS_A$ under translations and dilations,
        $Z'$ must contain a member of $\cS_A$, which, by
        construction, is not covered by $\cG$. This is a contradiction with our
        assumptions on $\cG$ and finishes the proof.
\end{proof}

Crucially, the proof depends on a certain smoothness condition on $r_{A}(n)$
(Equation \eqref{eq:rk_hypothesis} in the statement of the proposition),
which we will only be able to guarantee for an infinite sequence of values of
$n$. Since any sampling argument will forcefully have to relate $r_{A}(n)$ at
different values of $n$, such a hypothesis seems fundamentally necessary for our
approach to prove a supersaturation result of comparable strength.
We now prove such a smoothness condition for an infinite sequence. 
\begin{lemma}
        \label{lemma:inf_seq}
        If the extremal threshold function $r_A(n)$ satisfies
        \begin{equation}
                \label{eq:supra_behrend}
                r_A(n) \geq n^d\exp(-O(\log(n)^{\alpha}))
        \end{equation}
        for some $0 < \alpha < 1/2$, then
        there exists an infinite sequence $(n_i)_{i\geq 1}$ satisfying
        \[
                \frac{r_{A}(n_i)}{n_i^d} \geq \left(1-\frac{1}{\log \log
                        n_i}\right) \frac{r_{A}\left(n_ie^{-\log
        (n_i)^{1/2}}\right)}{n_i^de^{-d\log(n_i)^{1/2}}}.
        \]
\end{lemma}
\begin{proof}
      Suppose for a contradiction that the conclusion does not hold,
        meaning there exists $n_0$ such that
        \begin{equation}
\label{eq:bound_rk}
                 \frac{r_A(n)}{n^d} < \left(1-\frac{1}{\log \log
        n}\right)
        \frac{r_A\left(ne^{-\log(n)^{1/2}}\right)}{n^de^{-d\log(n)^{1/2}}}
\end{equation}
        for all $n \geq n_0$. For large enough $m$, define $m_0 = m \geq n_0^2$ and
        iteratively define the decreasing sequence $m_{i+1} = m_i
        e^{-\log(m_i)^{1/2}}$ while $m_i \geq m^{1/2}$. Let $m_t$ be the last
        value where this inequality holds, so that $m_t \geq m^{1/2}$ but
        $m_{t+1} < m^{1/2}$. Define also the sequence of
        densities $\delta_i = r_A(m_i)/m_i^d$ for $0 \leq i \leq t$, and notice
        that \eqref{eq:bound_rk} becomes
        \[
                \delta_i < \left(1-\frac{1}{\log\log m_i}\right) \delta_{i+1} \leq
                \exp\left(-1/\log\log(m_i)\right) \delta_{i+1}.
        \]
        Repeatedly applying this bound gives
        \[
                \frac{r_A(m)}{m^d} = \delta_0 \leq
                \exp\left(-\sum_{i=0}^{t-1}\frac{1}{\log\log(m_i)}\right)\delta_{t}.
        \]
        Bounding $\delta_{t}$ by $1$ and multiplying by $m^d$ on both sides
        gives
        \begin{equation}
\label{eq:final_bound_rk}
                r_A(m) \leq m^d \exp\left(-\frac{t}{\log\log(m)}\right).
        \end{equation}
        In order to lower bound $t$, notice that $m_i/m_{i+1}  \leq
        e^{\log(m)^{1/2}}$, so 
        \[
                m^{1/2} \leq \frac{m_0}{m_{t+1}} = \frac{m_0 \dots m_{t}}{ m_1
                \dots m_{t+1}}\leq e^{(t+1)\log(m)^{1/2}},
        \]
        from which we deduce that $t \geq \log(m)^{1/2}/3$. Plugging this into
        \eqref{eq:final_bound_rk} gives
        \[
                r_A(m) \leq m^d\exp\left(-\frac{\log(m)^{1/2}}{3\log\log m
                }\right)
        \]
        for large enough $m$, a contradiction with \eqref{eq:supra_behrend}
        since $\alpha < 1/2$.
\end{proof}

Putting together both lemmas, we obtain a concrete form of our supersaturation
result for an infinite sequence of $n$.
\begin{corollary}
        \label{cor:sequence}
        Provided $r_A(n)$ satisfies \eqref{eq:supra_behrend},
        there exists an infinite sequence $(n_i)_{i \geq 1}$ such that the
        following holds for every $i \geq 1$. A subset $B \subset [n_i]^d$ of size
        $|B| \geq (1+8/\log\log n_i) r_{A}(n_i)$ admits no cover $\cG$ of $\cS_{A}[B]$
        satisfying
        \begin{equation}
                w_q(\cG) \leq q|B| \quad \text{and} \quad q \geq
                \exp\left(-\frac{\log(n_i)^{1/2}}{4k}\right)
        \end{equation}
        with $|E| \geq 2$ for all $E \in \cG$.
\end{corollary}
\begin{proof}
        Let $(n_i)_{i \geq 1}$ be the sequence provided by Lemma
        \ref{lemma:inf_seq}. For a fixed $i \geq 1$, set $\delta = 8/\log \log
        n_i$ and $\ell = n_ie^{-\log (n_i)^{1/2}}$.  Condition
        \eqref{eq:rk_hypothesis} is guaranteed by Lemma \ref{lemma:inf_seq}, and
        $n \gg \ell \gg 1/\delta \gg 1$ is ensured by removing the first terms
        of the sequence $(n_i)$ if necessary. Applying Proposition
        \ref{prop:supersaturation_seq} gives that there exists no cover $\cG$ of
        $\cS_A[B]$ satisfying
        \[
                w_q(\cG) \leq q|B|,
        \]
        with $|E| \geq 2$ for all $E \in \cG$ and
        \[
                q \geq C\log(n_i)\log\log(n_i)^2
                \exp\left(-\frac{\log(n_i)^{1/2}}{2k}\right)
        \]
        for some constant $C > 0$ large enough. Again taking $n_i$ large enough
        by possibly suppressing some initial terms of the sequence, we may absorb
        the logarithmic terms into the exponential and conclude.
\end{proof}

With this in hand, we are ready to prove the main theorem of this section.

\begin{proof}[Proof of Theorem \ref{theo:inf_seq_bounds}]
        Let $(n_i)_{i \geq 1}$ be the sequence provided by Corollary
        \ref{cor:sequence}.
        Given large enough $n$ belonging to this sequence, set 
        $q = q(n) = \exp(-\log(n)^{1/2}/4k)$, 
        and apply Theorem \ref{theo:cs_containers} to $\cS_{A}[[n]^d]$, which has
        edge size at most $k$ and at least $2$, since we avoid trivial solutions
        to the system associated to $A$. The
        theorem provides  a family of containers
        $\cC$ satisfying
        \begin{itemize}
                \item The total number of containers is bounded by the number of
                        possible fingerprints of size $8k^2qn^d$, in other words
                        \[
                                \log(|\cC|) \lesssim qn^d \log(n) =
                                \exp\left(-\frac{\log(n)^{1/2}}{4k}\right) n^d \log(n)=
                                o(r_{A}(n)),
                        \]
                        where we used the lower bound on $r_{A}(n)$ from the
                        statement's hypothesis.
                \item Every $A$-free subset of $[n]^d$ is a subset of some $C
                        \in \cC$.
                \item For every $C \in \cC$, there exists a cover $\cG$ of
                        $\cS_{A}[C]$ satisfying \[ w_q(\cG) \leq q |C| \] with
                        $|E| \geq 2$ for all $E \in \cG$.
        \end{itemize}
        From Corollary \ref{cor:sequence} we deduce that $|C| \leq (1+8/\log\log
        n)r_{A}(n)$ for all $C \in \cC$. Since every $A$-free subset of
        $[n]^d$ is contained in some $C \in \cC$, it follows that the total
        number of $A$-free subsets in $[n]^d$ is at most
        \[
                \sum_{C \in \cC} 2^{|C|} \leq |\cC| 2^{(1+o(1))r_{A}(n)} \leq
                        2^{r_{A}(n)(1+o(1))},
        \]
        thus concluding the proof.
\end{proof}
\section{Bounds for all \texorpdfstring{$n$}{n}}
\label{sec:alln}
This section is dedicated to proving Theorem \ref{theo:kaps_alln}, which
bounds the number of $A$-free sets in $[n]^d$ by $2^{O(r_A(n))}$ for all
$n$. In the whole section, we fix dimension $d \geq 1$, the number of variables $k \geq 3$,
and $A \in \bZ^{m\times kd}$ a translation-invariant matrix with its associated
system of equations.

As in the previous section, the main ingredient of our proof is a
supersaturation result, phrased in the language of cheap covers. The proof of
this result is in the same spirit as the previous one. The key
difference is the fact that we are not allowed to pass to a subsequence of valid
$n$, and thus we cannot control the growth of $r_{A}(n)/n^d$ as we did in
Lemma \ref{lemma:inf_seq}. Instead of this, we use the following
supermultiplicativity property of $r_A(n)$, which is a straightforward
generalisation of a property observed by Ruzsa \cite[Theorem 6.3]{1993.ruzsa},
and follows from a product construction.
\begin{lemma}
\label{lemma:lower_multiplicative}
        The extremal threshold $r_A(n)$ satisfies
        \begin{equation}
                \label{eq:lower_multiplicative}
                r_A(m)r_A(n) \lesssim r_A(mn)
        \end{equation}
        for all $m, n \in \bZ_{>0}$.
\end{lemma}
\begin{proof}
        Let $B \subset[m]^d$ be an $A$-free set of size $r_A(m)$, and let $C
        \subset [n]^d$ be an $A$-free set of size $r_A(n)$ in $[n]^d$. Consider
        the product set $Z = B \times C \subset (\bZ^2)^d$. We claim
        that $Z$ is $A$-free with underlying group equal to $\bZ^2$.
        Indeed, let $z_1, \dots, z_k \in Z^k$  be a solution of
        \[
                A \cdot \vect(z_1, \dots, z_k) = 0.
        \]
        Since $z_i = (b_i, c_i)$ with $b_i \in B$ and $c_i \in C$,
        and the
        projection to a given coordinate of $\bZ^2$ is a linear operation,
        it follows that
        \[
                A \cdot \vect(b_1, \dots, b_k) = 0.
        \]
        Since $B$ is $A$-free, this implies that $b_1 = \dots = b_k$.
        Analogously, it also holds that $c_1 = \dots = c_k$, so $z_1 = \dots =
        z_k$,
        which is a trivial solution. In other words, $Z$
        is $A$-free, as we claimed. Finally, applying Corollary
        \ref{cor:freiman_product} gives a constant $s > 0$ depending on $A$ and
        an $A$-free set $Z' \subset [s mn]^d$ of size $|Z'| = |Z| = |B||C| =
        r_A(m)r_A(n)$. This in turn gives that $r_A(m)r_A(n) \leq r_A(smn)$, and
        an application of Lemma \ref{lemma:upper_multiplicative} finishes the proof.
\end{proof}

Compared to Lemma \ref{lemma:inf_seq}, supermultiplicativity as in
Lemma \ref{lemma:lower_multiplicative} gives
significantly weaker control on the growth of $r_A(n)/n^d$ and a
correspondingly weaker supersaturation result, from which we only deduce Theorem \ref{theo:kaps_alln} after iterated applications of the container theorem. As a side effect of these weaker bounds, we cannot afford the logarithmic penalty
introduced by taking grids with prime common difference in our previous
argument, so we take another route by shifting
the whole problem to the modular setting. Hence, we need that moving a given set
to the modular setting preserves members of $\cS_A$, which we encode
in the next lemma.
\begin{lemma}
        Given an integer $s > 0$, there exists a constant $C > 0$ such that, for
        every integer $n > 0$, there is a prime $p \leq Cn$ that makes the
        natural projection map
        \begin{align*}
                \rho \colon [n]^d &\to (\bZ/p\bZ)^d \\
                (x_1, \dots, x_d) &\mapsto (x_1, \dots, x_d) \mod p
        \end{align*}
        a Freiman $s$-isomorphism when restricted to its image.
        \label{lemma:freiman_mod}
\end{lemma}
\begin{proof}
        Using Bertrand's Postulate, we find a prime $s n \leq p \leq 2s n$. We
        claim that $\rho$ is then a Freiman $s$-isomorphism. Since the projection
        is linear, it is a Freiman homomorphism of any order. For the reverse
        direction, suppose that
        \[
                a_1 + \dots + a_s \equiv b_1 + \dots + b_s \mod p
        \]
        for $a_1, \dots, a_s, b_1, \dots, b_s \in [n]^d$. Componentwise, this
        means that $p \mid a_1^i + \dots + a_s^i - b_1^i - \dots - b_s^i$ for $1
        \leq i \leq d$. Since $|a_1^i + \dots + a_s^i - b_1^i - \dots - b_s^i| <
        ns$ and $p \geq ns$, this means that $a_1^i + \dots + a_s^i - b_1^i -
        \dots - b_s^i = 0$, so
        \[
                a_1 + \dots + a_s = b_1 + \dots + b_s,
        \]
        which concludes the proof.
\end{proof}

We may now prove our supersaturation result using these two lemmas.
\begin{proposition}
        \label{prop:supersaturation_alln}
        Let $n$ be an integer and $\lambda > 0$ satisfying $n \gg \lambda \gg
        1$. Suppose that
        \begin{equation}
                \label{eq:behrend_alln}
                        r_A(n) \geq n^d\exp(-O(\log(n)^{\alpha}))
        \end{equation}
                for some $0 < \alpha < 1$.
        Given $B \subset [n]^d$ a subset of
        size $|B| = \lambda r_A(n)$, there is no cover $\cG$ of $\cS_A[B]$
        satisfying
        \[
                w_q(\cG) \leq q|B|
        \]
        for $q = 1/\lambda^3$ and $|E| \geq 2$ for all $E \in \cG$.
\end{proposition}
\begin{proof}
        Suppose for contradiction there exists such a cover $\cG$, which we may
        assume to be a subfamily of $2^{B}$ and $k$-bounded. In a similar
        manner to the proof of Proposition \ref{prop:supersaturation_seq}, we
        will restrict $B$ to a random arithmetic progression and show that
        this intersection contains members of $\cS_A$ that are not covered by $\cG$.

        As we mentioned in the discussion above, it is convenient to do this in
        the modular setting. Let $s > 0$ be the integer provided by Lemma
        \ref{lemma:freiman_afree}, and $p$ the prime given by Lemma
        \ref{lemma:freiman_mod} so that the projection $\rho \colon [n]^d \to
        (\bZ/p\bZ)^d$ is a Freiman $s$-isomorphism when restricted to  its image. 
        We write $\tilde{B} = \rho(B)$
        for the projection of $B$ in $(\bZ/p\bZ)^d$. We also write $\tilde{\cG} =
        \rho(\cG)$ for the projection of the cover $\cG$.  Observe that,
        since $\rho$ is a bijection between $B$ and $\tilde{B}$, the collection
        $\tilde{\cG}$ has weight
        $w_q(\tilde{\cG}) = w_q(\cG)$.  We claim that $\tilde{\cG}$ is  a cover of  $\tilde{\cS}_A[\tilde{B}]
        \coloneqq
        \rho(\cS_A)[\tilde{B}]$. Indeed, given an element $e\in \tilde{\cS}_A[\tilde{B}]$, we have that $\rho^{-1}(e)\in \cS_A[B]$ on account of $\rho$ being a Freiman
        $s$-isomorphism for $[n]^d$ and Lemma \ref{lemma:freiman_afree}. Therefore, there exists some $g\in \cG$ with $g\subseteq \rho^{-1}(e)$ and so $e\supseteq \rho(g)\in \tilde{\cG}$. 
       Hence  it suffices to find a member
        of $\tilde{\cS}_A[\tilde{B}]$ uncovered by $\tilde{\cG}$ to reach a contradiction.

        Let $\mathbf{L}$ be a grid in $(\bZ/p\bZ)^d$ chosen uniformly at random
        among those of length $\ell = p/\lambda^{4k}$. In more formal terms, let
        \begin{equation}
                \label{eq:random_box}
                \mathbf{L} = \mathbf{b} + \mathbf{r} \cdot [\ell]^d,
        \end{equation}
        where the basepoint $\mathbf{b}$ is chosen uniformly at random in
        $(\bZ/p\bZ)^d$ and $\mathbf{r}$ is a scalar taken uniformly at random in
        $(\bZ/p\bZ)^*$.

        We will need to
        bound the number of possible outcomes of $\mathbf{b}$ and $\mathbf{r}$ that
        give $\mathbf{L}$ containing  
        two different points $x, y \in (\bZ/p\bZ)^d$. Without loss
        of generality, suppose that $x^1 \neq y^1$.
        If $x^1 = b^1 + r i$ for some
        $b^1, r \in \bZ/p\bZ$ and $i \in [\ell]$ and $y^1 = b^1 + r
        j$ for some $j \in [\ell]$ with $j \neq i$, then $r =
        (x^1-y^1)/(i-j)$ (here we use that $p$ is prime), so the value of
        $\mathbf{r}$ is determined by the relative position of the first
        coordinate of $x$ and $y$ in $\mathbf{L}$.
        Once we know $\mathbf{r}$, the value of $\mathbf{b}$ is determined by the
        relative position of $x$ in $\mathbf{L}$. There are a total of $\ell^d$
        choices for the relative position of $x$ and $\ell-1$ choices left for the
        relative position of the first coordinate of $y$, so
        there are at most $\ell^{d}(\ell-1)$ possible outcomes of $\mathbf{b}$ and
        $\mathbf{r}$ such that $\mathbf{L}$ contains $x,y$.
        On the other hand, there are a total of $p^{d}(p-1)$
        possible outcomes of $\mathbf{b}$ and $\mathbf{r}$. This tells us that
        \[
                \Pr(x, y \in \mathbf{L}) \leq \frac{\ell^d(\ell-1)}{p^d(p-1)} \leq
                \frac{\ell^{d+1}}{p^{d+1}},
        \]
        from which we deduce that
        \[
                \E\left[\left|\tilde{\cG} \cap 2^{\mathbf{L}}\right|\right] = \sum_{E
                \in \tilde{\cG}} \Pr(E \in \mathbf{L}) \leq \sum_{2 \leq i \leq
        k} \frac{w_q(\tilde{\cG})}{q^i} \left(\frac{\ell}{p}\right)^{d+1} \lesssim
q^{-(k-1)} |\tilde{B}| \left(\frac{\ell}{p}\right)^{d+1},
        \]
        where we used that $|E| \geq 2$ and that $w_q(\tilde{\cG}) \leq q|\tilde{B}|$.
        Plugging in $q = \lambda^{-3}$, $|\tilde{B}| = |B| = \lambda r_A(n)$, and $\ell
        =p/\lambda^{4k}$, we obtain that
        \begin{equation}
                \label{eq:upper_exp}
                \E\left[\left|\tilde{\cG} \cap 2^{\mathbf{L}}\right|\right] \lesssim
                \lambda^{-k} r_A(n)
        \left(\frac{\ell}{p}\right)^{d}.
        \end{equation}

        On the other hand, there are $\ell^{d}(p-1)$ values of $\mathbf{b}$ and
        $\mathbf{r}$ such that $\mathbf{L}$ contains a given
        $x \in (\bZ/p\bZ)^d$, as may be seen by choosing the relative position of $x$ in the
        grid and the value of $\mathbf{r}$. Therefore,
        \[
                \Pr(x \in \mathbf{L}) = \frac{\ell^d (p-1)}{p^d(p-1)} =
                \left(\frac{\ell}{p}\right)^d
        \]
        for every $x \in (\bZ/p\bZ)^d$. From this and $\eqref{eq:upper_exp}$, it
        follows that
        \[
                \E\left[\left|\tilde{B}\cap \mathbf{L}\right| - \left|\tilde{\cG}
                        \cap 2^{\mathbf{L}}\right|\right] 
                \gtrsim \left(\frac{\ell}{p}\right)^d |\tilde{B}| -
                C\lambda^{-k} r_A(n)
\left(\frac{\ell}{p}\right)^d,
        \]
        for some constant $C >0$ depending on $k$ and $d$. If $\lambda$ is large enough, since
        $|\tilde{B}|= \lambda r_A(n)$, we obtain
        \[
                \E\left[\left|\tilde{B}\cap \mathbf{L}\right| - \left|\cG \cap
                        2^{\mathbf{L}}\right|\right] \gtrsim
                \left(\frac{\ell}{p}\right)^d|\tilde{B}|.
        \]

        Hence, by a first moment argument there exists some grid $L$ of
        length $\ell$ that satisfies the previous lower bound. For every
        term of $\tilde{\cG} \cap 2^{L}$, we select an element and remove it from
        $\tilde{B}$ to obtain a new set $\tilde{B}' \subset L$ satisfying $|\tilde{B}'|
        \gtrsim (\ell/p)^d |\tilde{B}|$ and $\tilde{\cG} \cap 2^{\tilde{B}'} =
        \varnothing$, that is, such that $\tilde{\cG}$ covers no subset of
        $\tilde{B}'$. However, using \eqref{eq:lower_multiplicative} and
        \eqref{eq:behrend_alln} and the fact that $n \leq p \lesssim n$, we also have that
        \[
                |\tilde{B}'| \gtrsim \left(\frac{\ell}{n}\right)^d \lambda r_A(n)
                \gtrsim \left(\frac{\ell}{n}\right)^d \lambda
                r_A\left(\frac{n}{\ell}\right) r_A(\ell) \gtrsim \lambda
                e^{-O(\log(n/\ell)^{\alpha})} r_A(\ell),
        \]
        for some $0 < \alpha < 1$ and $n/\ell$ large
        enough, which follows from $\lambda \gg 1$. Plugging in $n/\ell \lesssim
        \lambda^{4k}$ and taking $\lambda$ large enough in the previous
        inequality guarantees that
        \[
                |\tilde{B}'| > r_A(\ell).
        \]

        Recalling the definition of $\mathbf{L}$ in \eqref{eq:random_box}, there
        exist $b \in (\bZ/p\bZ)^d$ and $r \in (\bZ/p\bZ)^*$ such that $L = b + r \cdot [\ell]^d$. We
        unfold the $\ell$-AP via the map
        \begin{align*}
                \varphi\colon L &\to [\ell]^d\\
                b + (r x^1, \dots, r x^d) &\mapsto (x^1, \dots, x^d).
        \end{align*}
        Since $\bZ/p\bZ$ is a field and $r \neq 0$, the map $\varphi$ is
        injective and preserves $|\varphi(\tilde{B}')| =
        |\tilde{B}'| > r_A(\ell)$, so there exist $x_1, \dots, x_k \in
        \varphi(\tilde{B}')$
        not all equal
        solving $A \cdot \vect(x_1, \dots, x_k) = 0$. Given that
        $\varphi^{-1}$ is
        an affine map, and the fact that the system associated to $A$ is invariant
        by translations and also by dilations on account of being homogeneous,
        we also have that $A \cdot \vect(\varphi^{-1}(x_1), \dots,
        \varphi^{-1}(x_k)) = 0$. In other words, $\{\varphi^{-1}(x_1), \dots,
        \varphi^{-1}(x_k)\} \in \tilde{\cS}_A[\tilde{B}]$ and is not covered by
        $\tilde{\cG}$, because $\tilde{B}'$ contains no member of $\tilde{\cG}$.
        This is a contradiction with $\tilde{\cG}$ being a cover of
        $\tilde{\cS}_A[\tilde{B}]$, and finishes our proof.
\end{proof}

We are now ready to prove our main goal, namely the following theorem that provides a
family of containers for $A$-free sets. We use an argument of Morris and
Saxton \cite{2016.morris.saxton} that iteratively applies Theorem
\ref{theo:cs_containers}, successively refining containers, until we obtain a
family of them of small enough size.
\begin{theorem}
        \label{theo:alln_containers}
        Provided $r_A(n)$ satisfies \eqref{eq:behrend_alln}, and given an
        integer $n \gg 1$, there exists a family of \emph{containers} $\cC =
        \cC(n, k) \subset 2^{[n]^d}$ satisfying:
        \begin{enumerate}[label=(\roman*)]
                \item The bound
                        \[
                                \log(|\cC|) \lesssim r_A(n)
                        \]
                        holds.
                \item Every container $C \in \cC$ satisfies
                        \[
                                |C| \lesssim r_A(n).
                        \]
                \item Every $A$-free set in $[n]^d$ is a subset of some
                        container $C \in \cC$.
        \end{enumerate}
\end{theorem}
\begin{proof}
        We will apply the container theorem to the hypergraph $\cH \coloneq \cS_A[[n]^d]$.
        Notice that a set in $[n]^d$ is
        $A$-free if and only if it is an independent set in $\cH$.
        We will build $\cC$ iteratively, successively applying the container
        lemma to obtain smaller containers at every step. More concretely,
        letting $\lambda$ be a constant large enough so that Proposition
        \ref{prop:supersaturation_alln} holds, we will find a sequence $\cC_1,
        \dots, \cC_t$ of families  of subsets of $[n]^d$, of length $t =
        \log_{2}\left(n^d/\lambda r_A(n)\right)$, satisfying the following
        properties:
        \begin{enumerate}[label=(\alph*)]
                \item For all $1 \leq i \leq t$ and $C \in \cC_{i}$, it holds that
                        \[
                                |C| \leq \lambda_i r_A(n),
                        \]
                        with $\lambda_i = \frac{n^d}{r_A(n)2^{i-1}}$ a dyadic
                        scale between $\frac{n^d}{r_A(n)}$ and $2\lambda$.
                        \label{item:condition_edges}
                \item For all $1 \leq i \leq t-1$,
                        \[
                                \log\left(\frac{|\cC_{i+1}|}{|\cC_{i}|}\right)
                                \lesssim \frac{r_A(n)}{\lambda_i}.
                        \]
                        \label{item:condition_containers}
                \item For all $1 \leq i \leq t$, every $A$-free set in
                        $[n]^d$ is a subset of some $C \in \cC_i$.
                        \label{item:condition_independent}
        \end{enumerate}

        Let us prove the existence of such a sequence.
        Set $\cC_1$ to contain only the whole set $[n]^d$, which satisfies
        conditions \ref{item:condition_edges} and
        \ref{item:condition_independent}. Assume
        conditions \ref{item:condition_edges} and
        \ref{item:condition_independent} hold for $\cC_i$, and let us define
        $\cC_{i+1}$. For every $C \in \cC_{i}$, we have two possibilities:
        \begin{itemize}
                \item If $|C| \leq \lambda_i r_A(n)/2  \leq \lambda_{i+1} r_A(n)$, we add it
        directly to $\cC_{i+1}$. 
\item If $|C| \geq \lambda_i r_A(n)/2 \geq \lambda_{i+1} r_A(n)$, apply Theorem
        \ref{theo:cs_containers} to the hypergraph $\cH[C]$ with $p = 1/\lambda_{i+1}^3$.
        We obtain a family of fingerprints $\cT \subseteq 2^{C}$ and a function
        $f$ satisfying the properties of the theorem. For every $T \in \cT$, add
        $f(T)$ to $\cC_{i+1}$. On account of property
        \ref{item:fingerprint_size} of Theorem \ref{theo:cs_containers} and
        writing $s=8k^2p|C|$, doing this we add at most
        \begin{equation}
                \label{eq:bound_ft}
                |\cT| \leq \sum_{j = 0}^{s} \binom{|C|}{j} \leq s \binom{|C|}{s}
                \leq s \left(\frac{|C|e}{s}\right)^s \lesssim p|C|
                p^{-s} \leq \lambda_{i+1}^{25k^2r_A(n)/\lambda_i^2}
        \end{equation}
        elements to $\cC_{i+1}$, where we used that $|C| \leq \lambda_i r_A(n)$.
        \end{itemize}
        
        For any element $f(T)$ added to
        $\cC_{i+1}$ in the second case, the choice of $p = 1/\lambda_{i+1}^3$ and
        property \ref{item:containers_cover} of Theorem \ref{theo:cs_containers}
        imply that
        \[
                |f(T)|  \leq \lambda_{i+1} r_A(n),
        \]
        since otherwise it would contradict Proposition
        \ref{prop:supersaturation_alln}.
        Finally, applying \eqref{eq:bound_ft} we have
        that
        \[
                \log\left(\frac{|\cC_{i+1}|}{|\cC_{i}|}\right) \lesssim \frac{r_A(n)}{\lambda_{i}}
        \]
        for large enough $\lambda_{i+1} \geq \lambda$, 
        so condition \ref{item:condition_containers} holds for $\cC_i$.

        We claim we may take $\cC$ in the theorem statement to be $\cC_t$. Let us verify it
        satisfies all the required properties. Condition
        \ref{item:condition_edges} on the sequence $(\cC_i)$ implies that $|C|
        \leq 2\lambda r_A(n)$ for
        every $C \in \cC$, and condition \ref{item:condition_independent} implies that
        every $A$-free subset of $[n]^d$ is a subset of some $C \in \cC$. As
        for the bound on $|\cC|$, iterating condition \ref{item:condition_containers} we have that
        \[
                \log(|\cC|) = \log(|\cC_1|) +\sum_{i = 1}^{t-1}
                \log\left(\frac{|\cC_{i+1}|}{|\cC_i|}\right) \lesssim \sum_{i = 1}^{t-1}
                \frac{r_A(n)}{\lambda_i} \lesssim \sum_{j = 0}^{t-1}
                \frac{r_A(n)}{2^{j}\lambda} \lesssim
                r_A(n),
        \]
        where we used that $|\cC_1| = 1$. This proves the desired conclusion.
\end{proof}

The main goal of this section is a straightforward consequence of the existence
of such a container family.
\begin{proof}[Proof of Theorem \ref{theo:main_alln}]
        For large enough $n > 0$, let $\cC$ be the family of containers $\cC$
        given by Theorem \ref{theo:alln_containers}. Since every $A$-free set in
        $[n]^d$ is a subset of some $C \in \cC$ of size at most $|C| \lesssim
        r_A(n)$, the number of $A$-free sets in $[n]^d$ is bounded by
        \[
                \sum_{C \in \cC} 2^{|C|} \leq |\cC| \cdot 2^{O(r_A(n))} = 2^{O(r_A(n))},
        \]
        where we used the bound on $|\cC|$ provided by Theorem
        \ref{theo:alln_containers}.
\end{proof}

\section{Applications}
\label{sec:applications}

In order to apply our main theorems to a certain family, we must show that it
may be described as the family of $A$-free sets for a certain
translation-invariant matrix $A \in
\bZ^{m\times kd}$, and exhibit
Behrend-type lower bounds for the extremal threshold. Let us first do so for the
family of $X$-free sets in $\bZ^d$.
\begin{lemma}
        \label{lemma:system_copiesx}
        Given integers $k, d > 0$ and a set $X \subset \bZ^d$ of size $|X|=k$, there exist $m 
> 0$ and a translation-invariant matrix $A \in \bZ^{m
        \times kd}$ such that any subset
        of $\bZ^d$ is $X$-free if and only if it is $A$-free.
\end{lemma}
\begin{proof}
        By invariance under translations, we may assume that $X$ contains the point $0
        = (0, \dots, 0) \in \bZ^d$. Write $x_2, \dots, x_{k} \in \bZ^d$ for the
        remaining points of $X = \{0, x_2, \dots, x_{k}\}$, and $(x_i^1,
        \dots, x_i^d) = x_i$ for their respective coordinates. 
        If $b_1, \dots, b_k \in \bZ^d$ form a copy of
        $X$, they may be written as
        \[
                (b_1, \dots, b_k) = (p, \dots, p) + \lambda \cdot (0, x_2,
                \dots, x_k), 
        \]
        for $p \in \bR^d$ and $\lambda \in \bR_{>0}$. Subtracting the
        corresponding equations and equating them coordinate-wise, it follows
        that
        \begin{equation}
                \label{eq:x_copy_b_coordinates}
                (b_i^s -b_1^s)x_j^t = (b_j^t - b_1^t)x_i^s,
        \end{equation}
        for all $2 \leq i, j \leq k$ and $1 \leq s, t \leq d$.

        We claim that the converse also holds, meaning that any $b_1, \dots,
        b_k \in \bZ^d$ satisfying \eqref{eq:x_copy_b_coordinates} 
        form a copy of $X$. 
        Indeed, 
        assume there exists $i$ such that $b_i \neq b_1$, since otherwise $b_1
        = \dots = b_k$ and they form a trivial copy of $X$.
        Let $s$ be a
        coordinate such that $b_i^s \neq b_1^s$. If $x_i^s = 0$, we deduce from
        \eqref{eq:x_copy_b_coordinates} that $x_j = 0$ for $j \not \in \{1, i\}$,
        a contradiction. Hence, we may divide by $x_i^s \neq 0$ in
        \eqref{eq:x_copy_b_coordinates} and obtain that
        \[
                (b_1,  \dots,  b_1) + \frac{b_i^s-b_1^s}{x_i^s} (0, x_2, \dots,
                x_k) = (b_1, b_2, \dots, b_k)
        \]
        on account of \eqref{eq:x_copy_b_coordinates}. This precisely means that
        $(b_1, \dots, b_k)$ form a copy of $X$. Therefore, the lemma follows
        by setting $A$ to be the matrix with $m=(d(k-1))^2$
        rows corresponding to every equation of the form
        $\eqref{eq:x_copy_b_coordinates}$, that is, with coefficients
        $x_j^t, -x_j^t, -x_i^s$ and $x_i^s$ in the columns corresponding to $b_i^s, b_1^s, b_j^t$ and $b_1^t$ respectively. Notice that \eqref{eq:x_copy_b_coordinates} is translation-invariant, because the coefficients in component $s$ are $x_j^t$ and $-x_j^t$, which add up to zero, and the same is true for component $t$. Hence, the matrix $A$ is also translation-invariant.
\end{proof}

We also need Behrend-type lower bounds for the multidimensional Szemerédi
problem. Since these are standard but it is not easy to find an appropriate
reference, we prove them in Appendix \ref{app:hypothesis_checking}, and simply
state them for now.
\begin{proposition}
        \label{prop:rankin_szemeredi}
        Given a set $X \subset \bZ^d$ of size $|X| \geq 2^{\ell}+1$, there exists a
        constant $C > 0$ such that
        \[
                r_X(n) \geq n^d \exp\left(-C\log(n)^{1/(\ell+1)}\right).
        \]
\end{proposition}
The case $\ell=1$ of the previous proposition was proved in \cite{2025.bhmnw}. The
counting versions of the multidimensional Szemerédi theorem are a corollary of
what we have proved so far.
\begin{proof}[Proof of Theorem \ref{theo:multi_sz_infseq} and
        \ref{theo:multi_sz_alln}]
        Let $A \in \bZ^{m \times kd}$ be the system of equations provided by
        Lemma \ref{lemma:system_copiesx}, so that $B \subset \bZ^d$ is $A$-free
        if and only if it is $X$-free. If $|X| \geq 5$, Proposition
        \ref{prop:rankin_szemeredi} gives
        \[
                r_X(n) \geq n^d \exp\left(-C\log(n)^{1/3}\right).
        \]
        By Theorem \ref{theo:inf_seq_bounds}, the number of $X$-free sets in
        $[n]^d$ is at most $2^{r_X(n)(1+o(1))}$. Analogously, applying Proposition
        \ref{prop:rankin_szemeredi} with $|X| \geq 3$ and Theorem
        \ref{theo:main_alln} gives Theorem \ref{theo:multi_sz_alln}.
\end{proof}

For the proofs of Theorem \ref{theo:random_systems_infseq} and
\ref{theo:random_systems_alln} we just need corresponding 
Behrend-type lower bounds, since they already deal with translation-invariant systems of
equations. The crucial input is the following result of Shapira
\cite{2006.shapira}.
\begin{proposition}
        \label{prop:rankin_random_system}
        Let $k, t \geq 3$ and $\ell > 0$ be integers
        satisfying   $\binom{t-2+2^{\ell}}{2^{\ell}} \leq k$.
        There are constants $c_1 = c_1(k)$ and $c_2 = c_2(k, h)$
        such that all but $c_1/h$ of the $(k, h)$-matrices $A \in \bZ^{(k-t+1)
        \times k}$ satisfy
        \[
                r_A(n) \geq n \exp\left(-c_2 \log(n)^{1/(\ell+1)}\right).
        \]
\end{proposition}
In the cited work, Shapira only proves the bound $r_A(n) \geq n \exp(-c_2
\log(n)^{1/2})$, as in the original Behrend construction, since this suffices
for his purposes. However, he does indicate how to obtain the stronger bounds in
Proposition \ref{prop:rankin_random_system}, which we do in Appendix
\ref{app:hypothesis_checking}. For now, we just note that Theorems
\ref{theo:random_systems_infseq} and \ref{theo:random_systems_alln} are an
immediate consequence of these bounds together with the $d=1$ case of Theorems
\ref{theo:inf_seq_bounds} and \ref{theo:main_alln}.

\section{Concluding remarks} \label{sec:final-remarks}
 Theorems \ref{theo:inf_seq_bounds} and  \ref{theo:main_alln}
improve our understanding of counting theorems of sets free of arithmetic
patterns, addressing the case of generic systems of equations and of arithmetic
progressions.
However, there remain several interesting open problems. The obvious one is
giving a full answer to the question of Cameron and Erd\H{o}s \cite{1990.cameron.erdos}.
\begin{question}
        Does Theorem \ref{theo:inf_seq_bounds} hold for $k = 3, 4$?
        Does it hold for all $n$?
\end{question}
We believe the answer to both questions is positive, and that Theorem
\ref{theo:inf_seq_bounds} itself provides ample evidence for the second one.
Regarding the first question, it is worth noting that the main weakness in our
proof of Theorem \ref{theo:inf_seq_bounds} does not lie only in the  smoothness
estimates for $r_k(n)$, such as those obtained in Lemma \ref{lemma:inf_seq}, but
also in the proof itself of our supersaturation result. In fact, even if we assume
the extremal threshold to have a specific form such as  $r_3(n) = n
\exp(-\sqrt{\log n})$, the random sampling argument in the proof of
supersaturation hits a natural obstacle, which is reflected in the fact that one
cannot guarantee hypothesis \eqref{eq:rk_hypothesis} in Proposition
\ref{prop:supersaturation_seq} for such a growth rate of $r_3(n)$.

A more modest goal than answering the question of Cameron and Erd\H{o}s in full
would be to obtain reasonable bounds on the constant in Theorem
\ref{theo:kaps_alln}, and particularly a bound that is not dependent on $k$. In
our proof, the iterated use of hypergraph containers is the main point where
this dependence plays a role.

The other main natural question is more open-ended and concerns clearing up the
picture for other systems of equations that do not satisfy Behrend or Rankin-type lower
bounds. A first question would be doing this in the case of a single equation.
\begin{question}
        Given an equation
        \[
                a_1 x_1 + \dots + a_k x_k = 0
        \]
        with $A = (a_1, \dots, a_k) \in \bZ^k$, when is the number of $A$-free
        sets $2^{r_A(n)(1+o(1))}$?
\end{question}
This already is a difficult task, since the behaviour of $r_A(n)$ is very far
from understood, and, as we already mentioned, there are some equations, such as
the one encoding Sidon sets, where the count is at least $2^{c r_A(n)}$ for some constant
$c > 1$.

\paragraph{\textbf{Acknowledgments}:}  A preliminary version of this work was presented at EUROCOMB 2025. We would like to thank Let\'icia Mattos for stimulating conversations on the topic of this paper.  

\bibliographystyle{abbrv}
\bibliography{main}
\appendix
\section{Rankin-type lower bounds}
\label{app:hypothesis_checking}
\subsection{Multidimensional Szemerédi}
\label{sec:multi_szemeredi}
The goal of this section is to prove Proposition \ref{prop:rankin_szemeredi}. As
we mentioned before, the proof is standard and uses Rankin's argument to find
large $k$-AP-free sets. However, we are not aware of the argument being carried
out explicitly in print, so we include it here for the sake of completeness.
Recall that our goal is the following. For every $X \subset \bZ^d$ with size
$|X| \geq 2^{\ell}+1$ and every large enough $n$, we must construct a set $A
\subset [n]^d$ that is $X$-free and has size
\begin{equation}
        \label{eq:rankin_szemeredi_size}
        |A| \geq n^d\exp\left(-C\log(n)^{1/(\ell+1)}\right),
\end{equation}
where $C > 0$ is a constant depending only on $X$ and $d$. Recall that our
definition of $X$-free excludes both positive and negative copies of $X$, but in
any case this is a more restrictive condition than the usual one, so as far as
lower bounds are concerned these are useful for both cases.

To do so, we first prove the existence of such a set in the one-dimensional
case, from which afterwards we deduce the larger-dimensional case via a
projection argument.
\begin{lemma}
        \label{lemma:rankin_multi_szemeredi}
        For $X \subset \bN$ of size $|X| \geq 2^{\ell}+1$ and large enough $n$, there
        exists an $X$-free set $A \subset [n]$ satisfying
        \[
        |A| \geq n\exp\left(-C\log(n)^{1/(\ell+1)}\right),
\]
\end{lemma}
The proof follows the presentation of Lacey and Laba \cite{2001.ll}. 
In
order to build large $X$-free sets in $[N]$, it is useful to first construct them in higher dimensional settings $[M]^t$ and then project them using an appropriate Freiman homomorphism. For that purpose, it is convenient to generalise the definition of copies of $X$ as follows.
\begin{definition}
        Let $t > 0$ be an integer.
        A polynomial copy of $X = \{x_1, \dots, x_k\} \subset \bN$ of degree $r$ in $\bR^t$ is a set
        \[
                \{P(x_1), P(x_2), \dots, P(x_k)\} \subset \bR^t,
        \]
        where $P \in \bR[x]^t$ is a tuple of $t$ polynomials of degree at most $r$ with
        real coefficients. We say that a polynomial copy is non-trivial if $P$
        is not constant, and that a set in $\bR^t$ is $(X, r)$-free if it contains no
        non-trivial polynomial copy of $X$ of degree $r$.
\end{definition}
We only define polynomial copies of $X$ for one-dimensional
configurations $X \subset \bN$ because that is all we need for the proof of
Lemma \ref{lemma:rankin_multi_szemeredi}.
Notice too that an $(X,1)$-free set in $\bR$ in particular is $X$-free in our
usual sense of copies of $X$.
Before giving the details of the proof of Lemma
\ref{lemma:rankin_multi_szemeredi}, let us give an overview of the three main
ideas: 
\begin{itemize}
        \item The first observation is that we may search for $X$-free sets in $[M]^t$, where
                $M$ and $t$ are parameters to be optimised, and then project
                this solution to $[N]$ in a manner that does not introduce
                copies of $X$.
        \item The second and most important point of the proof, drawing from
                Behrend's construction of 3-AP-free sets, is that a
                $t$-dimensional sphere intersects a polynomial of degree
                $2^{\ell-1}$ in at most $2^{\ell}$ points. Choosing the radius
                appropriately, we find a $t$-dimensional sphere with many
                integer points, thus finding a large $(X, 2^{\ell-1})$-free set
                in $[M]^t$.
        \item The final idea is to leverage the existence of
                $(X, 2^{\ell-1})$-free
                sets to build larger $(X, 2^{\ell-2})$-free sets. This is done by
                considering the union of spheres in $[M]^t$ whose radii form an
                $(X, 2^{\ell-1})$-free set. It turns out this union is then $(X,
                2^{\ell-2})$-free, and iterating this reasoning gives our final
                bound on $(X, 1)$-free sets.
\end{itemize}

In order to be able to project $X$-free sets in $[M]^t$ to $X$-free sets in $[N]$,
the following lemma will be useful.
\begin{lemma}
        \label{lemma:projection_behrend}
        Given $X \subset \bN$ and an integer $0 < r < |X|$,
        there exists an integer $K > 0$  such that the largest size of an
        $(X,r)$-free set in $[(KM)^t]$ is at least the one of any $(X,
        r)$-free set in $[M]^t$.
\end{lemma}
\begin{proof}
        Let $X = \{x_1, \dots, x_k\}$ be the elements of $X$. In order to use
        the properties we established for Freiman homomorphisms, it is convenient
        to describe polynomial copies of $X$ of degree $r$ as solutions to a
        certain system of equations. We begin with the case of a polynomial copy
        of $X$ of degree $r$ in $\bN$. Writing out the definition, integers
        $y_1, \dots, y_k \in \bN$ form a polynomial copy of $X$ of degree $r$ if
        \begin{equation}
                \label{eq:pol_copy}
                \begin{pmatrix}
                        1 & x_1 & \dots & x_1^r \\
                        1 & x_2 & \dots & x_2^r \\
                        \vdots & \vdots & & \vdots  \\
                        1 & x_k & \dots & x_k^r  \\
                \end{pmatrix} \cdot  
                \begin{pmatrix}
                        a_0 \\ \vdots \\ a_r
                \end{pmatrix}
                =
                \begin{pmatrix}
                        y_1 \\ \vdots \\ y_k
                \end{pmatrix}
        \end{equation}
        for some $a_0, \dots, a_r \in \bR$ coefficients of the corresponding
        polynomial. Write $V$ for the Vandermonde matrix on the left hand
        side, so that polynomial copies of $X$ lie in $\Ima V$ when $V$ is
        viewed as a linear map over $\bR$.
        Since $x_1, \dots, x_k$ are distinct and $k \geq r
        +1$, $V$ has rank $r+1$, and hence $\dim(\ker V^T) = k-r-1$. Let $v_1, \dots, v_{k-r-1} \in \bQ^{k}$ form a basis of
        $\ker V^T$, which we may choose to have rational coefficients on account
        of the entries of $V$ being integers. Notice that $v_i \cdot u = 0$ for
        any $u \in \Ima V$ for all $1 \leq i \leq k-r-1$, and, by a
        dimensionality argument, these are all possible solutions to the system
        formed by these equations. Since all the $v_i$ have rational coefficients, there
        exists $L \in \bN$ such that $Lv_i$ has integer coefficients for $1 \leq
        i \leq k-r-1$. Writing
        $A$ for the matrix whose $i$-th row is formed by $Lv_i$, by the previous
        discussion polynomial copies of $X$ of degree $r$ are precisely
        solutions to $A \cdot y = 0$ with $y^T = (y_1, \dots, y_k) \in \bN^k$.
        Therefore, a set in $\bN$ is $A$-free if and only if it is $(X,
        r)$-free. In particular, since $(1, \dots, 1)$ is a trivial polynomial copy of $X$, we have that $A \cdot (1, \dots, 1)^T = 0$ and $A$ is a translation-invariant matrix. As for polynomial copies in $\bN^t$, we obtain an equation as
        in \eqref{eq:pol_copy} for all $t$ components. 
        Applying the same reasoning as before for every component, we obtain that $y_1,
        \dots, y_k \in \bN^t$ form a polynomial copy of $X$ of degree $r$ if and
        only if $y^T = (y_1, \dots, y_k)$ satisfies $A \cdot y = 0$.

        Finally, let $K$ be the constant $s$ provided by Corollary \ref{cor:freiman_product}
        with $d=1$. Since being $(X, r)$-free is equivalent to being $A$-free
        both in $\bN$ and in $\bN^t$, Corollary \ref{cor:freiman_product}
        implies the desired statement.
\end{proof}

From the previous discussion, it should be clear that it is convenient to prove the
following generalization of Lemma \ref{lemma:rankin_multi_szemeredi}.
\begin{lemma}
        \label{lemma:inductive_rankin}
        For $X \subset \bN$ of size $|X| \geq 2^{\ell}+1$ and large enough $n$, given
        $0 \leq \ell' \leq \ell-1$, there
        exists an $(X, 2^{\ell'})$-free set $A \subset [n]$ of size
        \[
                |A| \geq  n\exp\left(-C\log(n)^{1/(\ell+1-\ell')}\right),
        \]
        where $C > 0$ is a constant depending on $X$.
\end{lemma}
\begin{proof}
        By possibly passing to a subset, we assume $X = \{x_1, \dots,
        x_{2^{\ell}+1}\}$.
        We argue by induction on $\ell'$ from $\ell' =\ell-1$ to $\ell'=0$. Suppose first that
        $\ell'=\ell-1$. In that case, let $M$ and $t$ be parameters to be optimised
        later, and consider
        \[
                B_r = \{(q_1, \dots, q_t) \in [M]^t \colon q_1^2 + \dots + q_t^2
                = r\},
        \]
        the integer points of a sphere of radius $\sqrt{r}$. Notice that any point in
        $[M]^t$ belongs to a single $B_r$ for $1 \leq r \leq tM^2$, so, by an
        averaging argument, there exists a particular $R \in [tM^2]$ such that $B=B_R$
        contains at least $M^t/(tM^2)$ points. We claim that $B$ is $(X,
        2^{\ell-1})$-free. Indeed, suppose that
        \[
                P(x_1), \dots, P(x_{2^{\ell}+1}) \in B
        \]
        for some $P = (P_1, \dots, P_t) \in \bR[x]^t$ with $P_i$ polynomials of
        degree at most $2^{\ell-1}$. Then
        \[
                P_1(x_i)^2 + \dots + P_t(x_i)^2 = R
        \]
        for all $1 \leq i \leq 2^{\ell}+1$. It follows that $P_1^2 + \dots + P_t^2
        - R$ is a
        polynomial of degree at most $2^{\ell}$ with $2^{\ell} + 1$ distinct roots, so it must be
        zero and $P_1^2 + \dots +P_t^2 = R$, which only happens if $(P_1,
        \dots, P_t)$ is constant. In other words, any polynomial copy of $X$ of
        degree $2^{\ell-1}$ is
        trivial, which proves our claim. Finally, appealing to Lemma
        \ref{lemma:projection_behrend}, there exists a constant $K$ such that
        $[(KM)^t]$ contains an $(X, 2^{\ell-1})$-free set of size $|B| = M^t/(tM^2)$.
        Setting $M = n^{1/t}/K$ and $t = \sqrt{\log(n)}$ gives a set of size at
        least
        \[
                \frac{M^t}{tM^2} \geq \frac{n}{K^t \sqrt{\log(n)} n^{2/t}} \geq
                \frac{n}{e^{(\log(K) + 3)\sqrt{\log(n)}}}
        \]
        in $[n]$, as desired.

        Having proved the base case, let us now carry out the induction step.
        Given $\ell' \leq \ell-2$, let $C_1$ be the constant provided by the lemma
        for the case $\ell'+1$. 
        Again, let $M$ and $t$ be parameters to be
        optimised. Our inductive hypothesis guarantees the existence of an $(X,
        2^{\ell'+1})$-free set $R \subset [M^2t]$ of size at least
        \[
                |R| \geq M^2t\exp\left(-C_1\log(M^2t)^{1/(\ell-\ell')}\right).
        \]
        For $z\in \bN$ with $z\leq 4tM^2$, we define
        \[
                B_z = \left\{ (q_1, \dots, q_t) \in [2M]^t \colon q_1^2 + \dots + q_t^2 \in
                z+R\right\},
        \]
        which we will later prove to be $(X, 2^{\ell'})$-free. Before doing that,
        we find a value of $z$ where $B_z$ contains many points.
        Notice that all $q \in [M, 2M]^t \subset [2M]^t$ satisfy
                $tM^2 \leq q_1^2 + \dots + q_t^2 \leq 4tM^2$, so there exist
                $|R|M^t$ triples $(q, r, z) \in [M, 2M]^t \times R \times
                [4tM^2]$ such that
        \[
                q_1^2 + \dots + q_t^2 - r = z
        \]
Averaging over all its possible values,
we obtain a particular $z$ such that $B = B_{z}$ has size
at least
\begin{equation}
        \label{eq:size_radius_set}
        |B| \geq \frac{|R|M^t}{4tM^2}.
\end{equation}
We now prove that indeed $B$ is $(X, 2^{\ell'})$-free. First,
observe that $z+R$ is still $(X, 2^{\ell'+1})$-free, because it is a
translation-invariant property.
        Suppose now that
        \[
                P(x_1), \dots, P(x_{2^{\ell}+1}) \in B
        \]
        for some $P = (P_1, \dots, P_t) \in \bR[x]^t$ with $P_i$ polynomials of
        degree at most $2^{\ell'}$. Then
        \[
                P_1(x_i)^2 + \dots + P_t(x_i)^2 \in z+R
        \]
        for all $1 \leq i \leq 2^{\ell}+1$. In other words, the set of $(P_1^2 +
        \dots + P_t^2)(x_i)$ form a polynomial copy of $X$ of degree
        $2^{\ell'+1}$ in $z+R$. Since $z+R$ is $(X, 2^{\ell'+1})$-free, $P_1^2 +
        \dots +P_t^2$ is constant on all points of $X$. Reasoning as in the base
        case, this implies that $(P_1, \dots, P_t)$ is constant, so any
        polynomial copy of $X$ of degree $2^{\ell'}$ in $B$ is trivial.

        Finally, again applying Lemma \ref{lemma:projection_behrend} gives a
        constant $K$ such that $[(MK)^t]$ contains an $(X, 2^{\ell'})$-free set
        of size $|B|$. Plugging in $t = \log(n)^{1/(\ell+1-\ell')}$ and
        $M = n^{1/t}/K$, and using \eqref{eq:size_radius_set} gives a set of size at least
        \begin{multline*}
                |B| \geq \frac{M^t|R|}{4tM^2} \geq
                \frac{n}{K^t 4}\exp\left(-C_1\log(M^2 t)^{1/(\ell-\ell')}\right)
                \geq \\ \geq
                n\exp\left(-(\log(K) + 3C_1)\log(n)^{1/(\ell+1-\ell')}\right)
        \end{multline*}
        in $[n]$, which finishes the inductive step and hence our proof.
\end{proof}

Once we are done with the case of one dimension, the general case follows by
projecting along a direction that does not remove any point in $X$ and pulling
back an $X$-free set in one dimension along this projection. The details follow. 
\begin{proof}[Proof of Proposition \ref{prop:rankin_szemeredi}]
        The case of $d=1$ is dealt with by Lemma \ref{lemma:rankin_multi_szemeredi},
        so we may assume that $d \geq 2$.
        By possibly passing to a subset, we may also take $X=\{x_1, \dots,
        x_{2^{\ell}+1}\}$.
        We first find $c \in \bN^d$ such that $(x_i-x_j) \cdot c \neq 0$ for all
        $1 \leq i < j \leq 2^{\ell}+1$. To do so, notice that the set of differences
        $x_i-x_j$ is finite, and hence the set of points $c$ where $(x_i-x_j)
        \cdot c = 0$ is a finite union of hyperplanes, which cannot cover
        $\bN^d$.
        Given such a $c$, define the projection
        \begin{align*}
                \pi \colon \bN^d &\to \bN\\
                x &\mapsto \pi(x)=c \cdot x.
        \end{align*}
        By construction, we also have that $\pi(x_i) = c \cdot x_i \neq c \cdot
        x_j = \pi(x_j)$ for $i \neq j$, so that $|\pi(X)| = |X|$.
        Write $K = c_1 + \dots + c_d$, which guarantees that $\pi\left([n]^d\right) \subset
        [Kn]$. By Lemma \ref{lemma:rankin_multi_szemeredi}, there exists a
        $\pi(X)$-free set $B \subset [Kn]$ of size as in 
        \eqref{eq:rankin_szemeredi_size} with constant $C_1>0$.    Since
        $\pi([n]^d) \subset [Kn]$, the set of triples 
        \[
                \left\{(y, b, s) \in [n]^d \times B \times [-Kn, Kn] \colon \pi(y) =
                b+s\right\}
                \]
        is of size at least $n^d |B|$. Averaging over all possible values
        of $s$, we obtain a particular $B' = B+s$ such that $A=\pi^{-1}(B')$
        has size at least
        \[
                |A| = \frac{n^d|B|}{2Kn} \geq \frac{n^{d}}{2K}\exp\left(-C_1
                        \log(Kn)^{1/(\ell+1)}\right).
        \]

        We claim that $A$ is $X$-free, which would finish the proof by taking $C$
        large enough. Indeed, suppose there exist $y \in \bR^d$ and $\lambda \in
        \bR_{>0}$ such that
        \[
                y + \lambda x_i \in A
        \]
        for all $1 \leq i \leq 2^{\ell}+1$. Then the values
        \[
                \pi(y+\lambda x_i) = c \cdot y + \lambda c \cdot x_i \in B'
        \]
        for $1 \leq i \leq 2^{\ell} +1$ form
        a copy of $\pi(X)$ in $B'$, which is $\pi(X)$-free on account of
        being a translation of a $\pi(X)$-free set, and hence $\lambda = 0$.
        This proves that all copies of $X$ in $A$ are trivial, as we claimed.
\end{proof}

\subsection{Systems of equations}
The goal of this section is to prove Proposition
\ref{prop:rankin_random_system}. Let $k,h, t$, and $\ell$ be integers as in the
proposition's statement, and write $m = k-t+1$ for the number of equations. As we already mentioned, we carry out an
argument of Shapira \cite{2006.shapira}, who only proved the bound $r_A(n) \geq
n\exp(- \Omega(\sqrt{\log n}))$ since it sufficed for his purposes, but
indicated that his argument together with Rankin's \cite{1960.rankin} gives the
claimed bound. Shapira's argument can be broken down in two parts:
\begin{enumerate}
        \item Solutions to almost all systems of $(k, h)$-equations admit a
                \emph{parametric representation} of the following form: there
                exist $p_1, \dots, p_k \in \bQ_{O_{h,k}(1)}^{t-2}$, meaning that the numerators
                and denominators of $p_i^j$ are bounded in
                terms of $h$ and $k$, and any solution to the system $z_1,
                \dots, z_{k} \in \bZ^r$, for $r \geq 1$, satisfies
                \begin{equation}
                        \label{eq:parametric_representation}
                        z_i = z_1 + p_{i}^1 (z_2 - z_1) + \dots +
                        p_i^{t-2} (z_{t-1} - z_1)
                \end{equation}
                for all $1 \leq i \leq k$.
                Furthermore, they admit a
                \emph{non-degenerate} parametric representation, meaning that the only
                $(t-2)$-variate polynomial $Q \in \bR[x_1, \dots, x_{t-2}]$ of
                total degree at most $2^{\ell}$ such that
                $Q(p_i^1, \dots, p_i^{t-2}) = 0$ for $1 \leq i \leq k$ is the
                null polynomial $Q=0$.
        \item Assuming that a system admits a non-degenerate parametric
                representation, use the argument of Rankin \cite{1960.rankin}
                together with $\eqref{eq:parametric_representation}$ to build a
                set without  non-trivial solutions to the system. More
                concretely, given $P = (p_1, \dots, p_k)$ as in
                \eqref{eq:parametric_representation}, we define a \emph{polynomial
                copy of degree $r$ of $P$} to be a set
                \[
                        \left\{Q(p_i^1, \dots, p_i^{t-2}) \colon 1 \leq i \leq k
                        \right\},
                \]
                for $Q \in \bR[x_1, \dots, x_{t-2}]$ a $(t-2)$-variate
                polynomial of total degree at most $r$. As one would expect, we
                say a copy is trivial if $Q$ is constant and a set is $(P, r)$-free
                if it only contains trivial polynomial copies of degree $r$.
                One may then use high-dimensional spheres and the non-degeneracy
                of $P$, as in Behrend's construction, to find a large $(P,
                2^{\ell-1})$-free set, and use Rankin's inductive scheme to
                eventually build a large $(P, 1)$-free set, which cannot contain
                non-trivial solutions to \eqref{eq:parametric_representation}.
\end{enumerate}
We closely follow Shapira's presentation of the argument, with some small
changes of notation for the sake of consistency with the rest of our paper. In
particular, note that our definition of the parametric representation is
slightly different from Shapira's, since we absorb the term $d$ appearing in
\cite[Claim 2.2]{2006.shapira} in the parameters $p_i$. Let us begin by
establishing the existence of a non-degenerate parametric representation as
claimed. First let us describe how the representation $P = (p_1, \dots, p_k)$ is
built. To do so, it is convenient to first reduce the problem to a particularly
simple form of matrices.
\begin{definition}
        We say an $m \times k$ matrix $A$ is \emph{diagonalised} if the
        only non-zero coefficients are those in the first $t-1$ columns and
        $A_{i, t-1+i}$ for $1 \leq i \leq m$.
\end{definition}
The advantage of diagonalised matrices is that it is easy to define their
parametric representation.
\begin{lemma}
        \label{lemma:parametric_diagonalised}
        Given a translation-invariant diagonalised matrix $A \in \bQ^{m \times
        k}$ with non-zero coefficients $A_{i, t-1+i} \neq 0$ for all $1 \leq i
        \leq m$, its set of solutions may be parametrised as in
        \eqref{eq:parametric_representation} with $p_1 = (0, \dots, 0)$, $p_i$
        equal to the $i^{th}$ unit vector for $2 \leq i \leq t-1$, and
\[
        p_i =\frac{1}{-A_{i-t+1, i}} \big(A_{i-t+1, 2}, \dots, A_{i-t+1,
                t-1}\big)
\]
for $t \leq i \leq k$.
\end{lemma}
\begin{proof}
        This corresponds to \cite[Claim 2.2]{2006.shapira}.
        The cases $i \leq t-1$ are immediate. For $i > t-1$,
        let $z^T = (z_1, \dots, z_k)$ with $z_j \in \bZ^r$ for some $r \geq 1$ be
        a solution of $A \cdot z = 0$, and let $c$ be the $(i-t+1)$-th row in $A$. Since $A$ is
        diagonalised, this implies that
        \[
                c_1 z_1 + \dots + c_{t-1} z_{t-1} + c_{i} z_{i} = 0.
        \]
        Isolating $z_{i}$ and using the fact that $-c_{i} = c_1 + \dots +
        c_{t-1}$ by translation invariance, we have that
        \[
                z_{i} = z_1 + \frac{c_2 (z_2 - z_1) +  \dots + c_{t-1}
                (z_{t-1}-z_1)}{-c_{i}},
        \]
        which implies the desired parametrisation.
\end{proof}
The previous lemma is only useful if most matrices are diagonalisable. To prove this,
we use the Schwartz-Zippel lemma in the following form (see \cite[Lemma
3.2]{2006.shapira}).
\begin{lemma}
        \label{lemma:zippel}
        Let $F$ be an arbitrary field, and let $f = f(x_1, \dots, x_b)$ be a
        non-zero polynomial in $F[x_1, \dots, x_b]$ where the degree in each
        variable is at most $r$. If $H$ is a subset of $F$ with $|H| > r$, then
        there are at most $|H|^b - (|H|-r)^b \leq br|H|^{b-1}$ assignments $x_1
        \in H, \dots, x_b \in H$ such that $f(x_1, \dots, x_b)=0$.
\end{lemma}
\begin{claim}
        \label{claim:diagonalise}
        There exists $c=c(k)>0$ and $R=R(k)\in \bN$ such that all but $c/h$ of $(k, h)$-matrices $B \in \bZ^{m \times k}$,
        there exists a diagonalised matrix $A \in \bQ_{O_{h, k}(1)}^{m\times k}$ such that the
        respective associated systems have the same solutions.
        Furthermore, if we regard the coefficients of $B$ as unknowns,
        the coefficients of $A$ may be expressed as rational
        functions with degree at most $R$ in terms of those of $B$.
\end{claim}
This corresponds to \cite[Claim 3.1 and item (i) in Lemma
        3.3]{2006.shapira}, where it is proved in more detail.
\begin{proof}
        Consider the coefficients of $B$ as unknowns. Using Gaussian elimination,
        one can transform $B$ to a matrix $A$ in diagonalised form whose
        coefficients are rational functions of those of $B$ with degree at most
        $R = R(k)$. If one of the denominators of $A_{i, j}$ was identically
        zero, the Gaussian elimination process would not be well defined for any
        values of $B$. This is not the case, since it is well defined, for
        example, when $B$ is already diagonalised.

        We use the Schwartz-Zippel lemma to show that $A$ is well-defined for
        most $(k, h)$-matrices $B$. Indeed, let $f$ be the product of all the
        denominators in $A_{i, j}$, so it is a non-zero polynomial in the
        coefficients of $B$ where every variable has degree at most $mkR \leq
        k^2R$. Recall that the coefficients of a $(k, h)$-matrix lie all in
        $[-h, h]$.
        Therefore, applying Lemma 
        \ref{lemma:zippel} with $f$, $b = m(k-1)$ (notice that one coefficient
        of every row in a $(k, h)$-matrix may be written in terms of the
        others) and $H = [-h, h]$ implies that at most
        $mk^3R (2h+1)^{m(k-1)-1}$ $(k, h)$-matrices $B$ satisfy $f = 0$. Note
        that there are at least $(h/k)^{m(k-1)}$ $(k, h)$-matrices, as may be seen
        by setting arbitrarily all but the last column of $B$ to values in
        $[h/k]$, and defining the last column to guarantee that all rows add up
        to zero. These two estimates imply that $A$ is well-defined for
        all but $c(k)/h$ of $(k, h)$-matrices.
\end{proof}

We will build the parametric representation of a given matrix by first
diagonalising it and then taking the representation as in Lemma
\ref{lemma:parametric_diagonalised}, whenever this is well-defined. This gives a
non-degenerate representation under the hypotheses on $k$ and $t$ of Proposition
\ref{prop:rankin_random_system}.
\begin{lemma}
        \label{lemma:parametric_representation}
        There exists $c=c(k)>0$ such that all but $c(k)/h$ of $(k, h)$-matrices $B \in \bZ^{m \times k}$ admit a
        non-degenerate parametric representation.
\end{lemma}
\begin{proof}
        Again, regard the coefficients of $B$ as unknowns, and write $A = A(B)$ for the
        diagonalised matrix obtained in Claim \ref{claim:diagonalise}, where
        coefficients are rational functions of the coefficients of $B$ and
        degree bounded by $R(k)$, and are well-defined for all but $c_1(k)/h$ of
        $(k, h)$-matrices $B$. We define the parametric representation to be the
        $p_i$ as defined in Lemma \ref{lemma:parametric_diagonalised} for $A$.
        Notice these are well defined if both $A$ is well defined and $A_{i-t+1,
        i}$ does not evaluate to $0$ for $t \leq i \leq k$. The coefficient
        $A_{i-t+1, i}$ cannot be identically zero as a function of those of $B$,
        as one may see by evaluating $B$ at a matrix that is already
        diagonalised and has non-zero entries. Since $A_{i-t+1, i}$ is a
        rational function of degree at most $R(k)$, applying Lemma
        \ref{lemma:zippel} implies that less than $c_2(k) h^{m(k-1)-1}$ of the
        $(k, h)$-matrices $B$ satisfy $A_{i-t+1, i} = 0$ for some $t \leq i \leq
        k$.

        Therefore, we have seen that the parametric representation $p_1 = p_1(B), \dots,
        p_k = p_k(B)$ is well defined for all but $(c_1+c_2)(k)/h$ of $(k,h)$-matrices
        $B$. Finally, let us see that it is a non-degenerate representation for
        almost all matrices. Write $m' = \binom{t-2+2^{\ell}}{2^{\ell}} -
        t + 1 \leq m$ and $k'=t-1+m'$. We will show that the tuples of the parametric representation
        corresponding to the first $m'$ rows of the diagonalised matrix $A$,
        that is, the tuples $p_1, \dots, p_{k'} \in \bQ^{t-2}$, cannot all be
        zeros of a given non-null $(t-2)$-variate polynomial $Q \in \bR[x_1,
        \dots, x_{t-2}]$, for all but a negligible amount of possible values of
        $B$. Indeed, let $\cI = \{(\alpha_1, \dots, \alpha_{t-2}) \in \bZ_{\geq
                0}^{t-2}
                \colon \alpha_1 + \dots + \alpha_{t-2} \leq
        2^{\ell}\}$ be the family of multi-indices of total degree at
        most $2^{\ell}$, so that $Q$ may be written as
        \[
                Q(x) = \sum_{\alpha \in \cI} c_{\alpha} x^{\alpha},
        \]
        where $x = (x_1, \dots, x_{t-2})$ and $x^{\alpha} = x_1^{\alpha_1} \dots
        x_{t-2}^{\alpha_{t-2}}$. Consider the system of equations that
        imposes $Q(p_1) =  \dots = Q(p_{k'}) = 0$. More concretely,
        define the $|\cI| \times |\cI|$ matrix $M = M(B)$ with columns
        indexed by $\cI$ such that $M_{i, \alpha} = p_i^{\alpha}$ for $1 \leq i
        \leq t-1+m' = |\cI| = \binom{t-2+2^\ell}{2^{\ell}}$ and $\alpha \in \cI$. Note that $Q$ is zero
        at $p_1, \dots, p_{k'}$ if and only if $M_{i, \alpha} c = 0$, where $c$
        is the vector of coefficients of $Q$ indexed by $\cI$. If we show that
        $M$ is invertible, then the only solution is $c = 0$, in other words,
        the only polynomial with zeros in $p_1, \dots, p_{k'}$ is identically zero.
        Hence, the statement of the lemma is reduced to the following claim.
        \begin{claim*}
                For all but $c_3(k)/h$ of $(k, h)$-matrices $B \in \bZ^{m\times
                k}$, the matrix $M$ has non-zero determinant.
        \end{claim*}
        By Lemma \ref{lemma:parametric_diagonalised}, we know that $p_i^j$ is a
        rational function of degree one on the coefficients of $A$. Therefore,
        every entry of $M$ may be regarded as a rational function of degree at
        most $2^{\ell}$ in terms of the coefficients of $A$. Since there are
        at most $k$ rows in $M$, its determinant in terms of the coefficients of
        $A$ is a rational function where every variable has degree $O_{k}(1)$, and by Claim \ref{claim:diagonalise}, the same is true when
        regarded as a function of the coefficients of $B$.

        Let us see the determinant of $M$ is not identically zero in terms of
        $B$. To do so, assume $B$ is already diagonalised, so $A = B$, and that
        $A_{i, t-1+i} = -1$ for $t \leq i \leq k$, so that $p_i = (A_{i-t+1,2}
        \dots, A_{i-t+1, t-1})$ may be arbitrarily set for $i > t-1$.
        We order the columns of $M$ in such a way that the index $(0, \dots, 0)$
        corresponding to constant terms in the polynomial is at the first
        position. In this case, the first row of $M$ is the vector $(1, 0,
        \dots, 0)$. This shows that we may ignore the first row and first column
        of $M$ to compute its determinant. As for the rest, since $p_i$ is the
        $i$-th unit vector for $2 \leq i \leq t-1$, there exist a set of columns
        $j_2, \dots, j_{t-1}$, such that the minor corresponding to rows $2,
        \dots, t-1$ and columns $j_2, \dots, j_{t-1}$ form a permutation matrix.
        If we expand the determinant of $M$ using the Leibniz formula for
        determinants, we obtain a sum of monomials depending on coefficients of
        $A$. Since every possible combination of $A_{i-t+1, j}$ in rows $2 \leq
        i$ determines the column uniquely, monomials involving columns $j_2,
        \dots, j_{t-1}$ at rows $2, \dots, t-1$ have coefficient $\pm 1$ and are
        not canceled by other monomials, so the determinant is not identically
        zero and has degree $O_k(1)$. Finally, an application of Lemma
        \ref{lemma:zippel} proves the claim and hence finishes our proof.
\end{proof}

As we mentioned at the beginning of the section, if a system admits a
non-degenerate parametric representation, we can then use the arguments of
Rankin to produce a large $A$-free set. We only sketch the proof, since it
follows the previous section closely. It is convenient to extend the definition
of polynomial copies of $P$ to the setting of $\bZ^s$. More concretely, given $P
= (p_1, \dots, p_k)$ as in \eqref{eq:parametric_representation}, we define a
\emph{polynomial copy of degree $r$ of $P$} in $\bZ^s$ to be a set
\[
        \left\{Q(p_i^1, \dots, p_i^{t-2}) \colon 1 \leq i \leq k
        \right\} \subset \bZ^s
\]
for $Q \in \bR[x_1, \dots, x_{t-2}]^s$ an $s$-tuple of $(t-2)$-variate
polynomials of total degree at most $r$. Sets that are $(P, r)$-free in $\bZ^s$
may be related to those in $\bZ$ by the following lemma.
\begin{lemma}
        \label{lemma:projection_behrend_parametric}
        Given a parametric representation $P \subset (\bQ^{t-2})^k$ as in
        \eqref{eq:parametric_representation} and an integer $0 < r$,
        there exists an integer $K > 0$  such that the largest size of a
        $(P,r)$-free set in $[(KM)^s]$ is at least the one of any $(P,
        r)$-free set in $[M]^s$.
\end{lemma}
\begin{proof}
        We omit the proof, since it is very similar to the one of Lemma
        \ref{lemma:projection_behrend}.
\end{proof}

\begin{lemma}
        \label{lemma:inductive_ranking_parametric}
        Let $A \in \bZ^{m \times k}$ admit a non-degenerate parametric
        representation $P = (p_1, \dots, p_k)$ of its associated system, and let
        $0 \leq \ell' \leq \ell-1$ be an integer.
        There exists a $(P, 2^{\ell'})$-free set $A \subset [n]$ of size
        \[
                |A| \geq  n\exp\left(-C\log(n)^{1/(\ell+1-\ell')}\right),
        \]
        where $C > 0$ is a constant depending on $A$.
\end{lemma}
\begin{proof}
        The proof is essentially the same as the one of Lemma
        \ref{lemma:inductive_rankin}, with the caveat that we are now concerned
        with copies of $P \subset (\bQ^{t-2})^{k}$ instead of copies of $X \subset \bN$.

        We induct on $\ell'$ from $\ell' = \ell-1$ to $\ell'=0$. 
        Let $M$ and $s$ be integer parameters to be optimised later,
        and consider
        \[
                B_r = \{(b_1, \dots, b_s) \in [M]^s \colon b_1^2 + \dots + b_s^2
                = r\}
        \]
        for integer $r > 0$. We claim that $B_r$ is $(P, 2^{\ell-1})$-free.
        Indeed, suppose there exist $Q_1, \dots, Q_s \in \bR[x_1, \dots,
        x_{t-2}]$ which are  $(t-2)$-variate polynomials of total degree at most
        $2^{\ell-1}$ such that $(Q_1(p_i), \dots,
        Q_s(p_i)) \in B_r$ for all $1 \leq i \leq k$, where $Q_j(p_i)$ is
        shorthand for $Q_j(p_i^1, \dots, p_i^{t-2})$. Then
        \[
                Q(p_i^1, \dots, p_i^{t-2}) \coloneq \sum_{1 \leq j \leq s}
                Q_j(p_i^1, \dots, p_i^{t-2})^2 = r,
        \]
        for all $1 \leq i \leq k$, so $Q - r$ is a $(t-2)$-variate polynomial
        of total degree  at most$2^\ell$
        that is zero on all $p_i$. Since $P$ is non-degenerate, $Q=r^2$ is a
        constant. Note that $Q$ is a sum of squares, which implies that $Q_j$
        is a constant for all $1 \leq j \leq s$, so it is in fact a trivial copy. By an
        averaging argument as in Lemma \ref{lemma:inductive_rankin}, there
        exists an $r_0$ such that $|B_{r_0}| \geq \frac{M^s}{sM^2}$. Applying Lemma
        \ref{lemma:projection_behrend_parametric}, there exists a constant $K$ such that
        $[(KM)^s]$ contains a $(P, 2^{\ell-1})$-free set of size $|B_{r_0}|$, and setting $M=n^{1/s}/K$ and
        $s = \sqrt{\log(n)}$ proves the base case.

        For the inductive step, let $M$ and $s$ be again integers to be
        optimised, and take a $(P, 2^{\ell'+1})$-free set $R \subset [M^2s]$ of
        size $|R| \geq M^2s\exp\left(-C_1\log(M^2s)^{1/(\ell-\ell')}\right)$ for
        a constant $C_1 > 0$. Reasoning as in Lemma
        \ref{lemma:inductive_rankin}, there exists an integer $0 \leq z \leq
        4sM^2$ such that the set
        \[
                B = \left\{ (q_1, \dots, q_s) \in [2M]^s \colon q_1^2 + \dots + q_s^2 \in
                z+R\right\},
        \]
        has size at least $|B| \geq \frac{|R|M^s}{4sM^2}$. The set $B$ is
        $(P, 2^{\ell'})$-free. Indeed, suppose there are polynomials $Q_1,
        \dots, Q_s \in \bR[x_1, \dots, x_{t-2}]$ of degree at most $2^{\ell'}$ such that
        $(Q_1(p_i), \dots, Q_s(p_i)) \in B$ for all $1 \leq i \leq k$. Then
        $(Q_1^2 + \dots + Q_s^2)(p_i)$ for $1 \leq i \leq k$ forms a polynomial
        copy of degree $2^{\ell'+1}$ in $z+R$, which by hypothesis must be trivial.
        This again implies that $Q_1, \dots, Q_s$ are all constant and the copy of $P$ must
        be trivial. Optimizing the parameters as in Lemma
        \ref{lemma:inductive_rankin} gives the result.
\end{proof}

Proposition \ref{prop:rankin_random_system} follows from
putting together Lemma \ref{lemma:parametric_representation}, Lemma
\ref{lemma:inductive_ranking_parametric}, and the observation that if $A$ admits
a non-degenerate parametric representation $P$ then a $(P, 1)$-free set is an
$A$-free set on account of \eqref{eq:parametric_representation}.
\end{document}